\documentclass[11pt]{article}

\usepackage{mathtools}
\usepackage{amsmath}
\usepackage{amsthm}
\usepackage{amssymb}
\usepackage{mathrsfs}
\usepackage{graphicx}
\usepackage{upgreek}
\newcommand\tbbint{{-\mkern -16mu\int}}

\newcommand\dbbint{{-\mkern -19mu\int}}

\newcommand\bbint{
{\mathchoice{\dbbint}{\tbbint}{\tbbint}{\tbbint}}
}

\def\ot{{\, \otimes \,}}
\def\div{{\,\rm div \,}}

\def\id{{\,\rm id \,}}

\def\sym{{\,\rm sym \,}}

\def\dist{{\,\rm dist \,}}

\def\supp{{\,\rm supp \,}}

\def\+M{{\,\rm M^{n\times n}_+ \,}}
\def\tr{{\,\rm tr \,}}

\def\diag{{\,\rm diag \,}}
\def\supp{{\,\rm supp \,}}

\def\E{{\cal E}}
\def\W{{\cal w}}

\def\Q{{\cal Q}}

\def\E{{\cal E}}

\def\W{{\cal W}}

\def\H{{\cal H}}

\def\Q{{\cal Q}}
\def\M{{\cal M}}

\def\ka{{\kappa}}

\newfont{\Blackboard}{msbm10 scaled 1200}

\newfont{\roma}{cmr10 scaled 1200}

\def\D{{\cal D}}

\def\<{{\langle}}
\def\>{{\rangle}}

\def\Ga{\Gamma}

\def\a{\alpha}
\def\b{\beta}
\def\om{\omega}
\def\Om{\Omega}

\numberwithin{equation}{section}
\newtheorem{thm}{{}\hskip\parindent Theorem}[section]
\newtheorem{lem}{{}\hskip\parindent Lemma}[section]

\def\dsum{\displaystyle\sum}

\def\dint{\displaystyle\int}
\def\dfrac{\displaystyle\frac}

\def\be{\begin{equation}}
\def\ee{\end{equation}}
\def\beq{\arraycolsep=1.5pt\begin{eqnarray}}
\def\eeq{\end{eqnarray}}

\def\n{\vec{n}}
\title{Rigidity Estimate for Hyperbolic Shells and its
 Application in $\Ga$-limit Theory}
\date{}
\author{
Liang-Biao Chen and Peng-Fei Yao\\[0.2cm]
\nonumber
School of Mathematics and Statistics\\\nonumber
Key Laboratory of Complex Systems and Data Science of Ministry of
Education\\\nonumber
Shanxi University, Taiyuan, 030006,
China\\\nonumber
e-mail: pfyao@iss.ac.cn}

\begin{document}
\maketitle
\begin{abstract}
This paper establishes novel rigidity estimates for hyperbolic shells (surfaces with negative Gaussian curvature) and applies them to derive the \(\Gamma\)-limit of thin elastic shells. We prove a nonlinear rigidity estimate for \(H^1\) deformations on the mid-surface, and a  nonlinear rigidity estimate for hyperbolic shells with clamped lateral boundary. The latter yields the optimal exponent \(h^{-4/3}\). As the main application, we characterize the \(\Gamma\)-limit of the nonlinear elastic energy for clamped hyperbolic shells across all scaling regimes \(\beta \in [0,2) \cup (8/3,\infty)\).
\end{abstract}
\section{Main Result}
Let $M\subset \mathbb{R}^3$ be a surface in $\mathbb{R}^3$ with unit normal vector $\n$, second fundamental form $\Pi=\nabla\n$ and Gaussian curvature $\ka<0$. Let $\<,\>$ be the metric on $M$ induced by the dot product $\<,\>$ in $\mathbb{R}^3$. Let $\nabla$ and $D$ be the Riemannian connections of $(\mathbb{R}^3,\<,\>)$ and $(M,\<,\>)$, respectively. 

 Local coordinate system $\a:\om\subset\mathbb{R}^2\to M$ is called asymptotic, if 
\beq
\Pi(\a_{x_i},\a_{x_i})=0,\quad i=1,2.\label{Pi}
\eeq
Let $S$ be a parametrized subsurface of $M$. 

{\bf Assumption (M):} $S=\a(\om)$ is a minimal surface with negative Gaussian curvature, where $\a:\om\to M$ is an  asymptotic coordinate system, where $\om\subset\mathbb{R}^2$ is a compact subset of $\mathbb{R}^2$ with piecewise smooth boundary.

{\bf Assumption (R):} $S=\b([0,1]^2)$ is a ruled surface with negative Gaussian curvature, where $\b(s,t)=\sigma(s)+t\delta(s)$, $(s,t)\in [0,1]^2$.

\subsection{Rigidity Theorems}

\begin{thm}\label{thm2d}
Suppose {\bf Assumption (M)} or {\bf Assumption (R)} holds. Then there exists a constant $C>0$ such that for any $y\in H^1_0(S,\mathbb{R}^3)$, there holds 
\beq
\|W\|_{L^2(S)}
+\|w\|_{(H^1(S))'}
+\|\sym DW+w\Pi\|_{L^1(S)}
\le C\|(u_*^Tu_*)^\frac12-\id\|_{L^2(S)},
\label{2d}
\eeq
where $w=\<y,\n\>$, $W=y-w\n$ and $u_*^Tu_*$ is the metric on $S$ induced by $u(x)=x+y(x)$, $x\in S$.

\end{thm}

Define the shell $S_h$ with midsurface $S$ and thickness $h>0$ by 
\[
S_h=\{
x+t\n(x):x\in S,|t|\le h/2\}, 
\]
lateral boundary $(\partial S)_h$ of $S_h$ by 
\[
(\partial S)_h=\{
x+t\n(x):x\in \partial S,|t|\le h/2\}, 
\]
and the space of displacement on $S_h$ fixing $(\partial S)_h$ by 
\[
\H^1_0(S_h,\mathbb{R}^3)
=\{y\in H^1(S_h,\mathbb{R}^3):y|_{(\partial S)_h}=0\}.
\]

\begin{thm}\label{thm3d}
Suppose {\bf Assumption (M)} or {\bf Assumption (R)} holds. Then there exist $h_0>0$ sufficiently small and constant $C>0$ such that for any $h\in(0,h_0)$ and $y\in \H^1_0(S_h,\mathbb{R}^3)$, there holds 
\beq
\|\nabla y\|^2_{L^2(S_h)}
\le Ch^{-\frac43}\|\dist(\nabla y+I,SO(3))\|^2_{L^2(S_h)}.
\label{3d}
\eeq

\end{thm}

\subsection{$\Ga$-Limits of Clamped Hyperbolic Thin Shells}
Define 
\[
\Q_3(F)
=\nabla^2\W(I_3)(F,F)\quad \mbox{for all}\quad F\in\mathbb{R}^{3\times3},
\]
and 
\[
\Q_2(x,A)
=\min\{\Q_3(F):F_{\tan}=A,\ F\in \mathbb{R}^{3\times3}\}\quad \mbox{for all}\quad A\in T^2_xS,\quad x\in S,
\]
where $F_{\tan}\in T^2_xS$ is the ristriction of $F$ on $T_xS$, i.e.,
\[F_{tan}(\tau,\mu)=\<F\tau,\mu\>\quad \mbox{for all}\quad \tau,\mu\in T_xS.
\]
Define $I_\b:L^2(S,T)\times (H^1(S))'\to[0,\infty]$ as 
\[
I_\b(W,w)
=
\begin{cases}
\frac12\dint_S\Q_2(\sym DW+w\Pi)dx,&\quad\mbox{if}\quad \b\in(\frac83,\infty)\\
\inf\left\{\dfrac12\dint_S\Q_2(\sym DW+w\Pi+\mu)dx:\ \mu\in\M(S,T^2_+S)
\right\},&\quad\mbox{if}\quad \b\in[0,\frac83).
\end{cases}
\]
Fix $h_0>0$ sufficiently small such that mapping
\[\pi:S_{h_0}\to S,\quad x+t\n(x)\mapsto x\]
is well-defined on $S_{h_0}$. Let $U=S\times[-1/2,1/2]$ be the product manifold of Riemannian manifolds $S$ and $[-1/2,1/2]$ equipped with metric
\[
\<(X_1,t_1\partial_t),(X_2,t_2\partial_t)\>=\<X_1,X_2\>+t_1t_2,\quad\mbox{for all}\quad 
(X_i,t_i)\in T_xS\times \mathbb{R}.
\]
Then the Riemannian connection $\D$ on $U$ takes the form 
\beq
\D_{X_1+t_1\partial_t}(X_2+t_2\partial_t)
=D_{X_1}X_2+t_1(\partial_tt_2)\partial_t.\label{connectionU}
\eeq
Define 
\[
\pi:S\times[-1/2,1/2]
\to S,\quad (x,t)\mapsto x,
\]
and
\[
\pi_h:S\times[-1/2,1/2]
\to S_h,\quad (x,t)\mapsto x+th\n(x).
\]
\begin{thm}\label{Gamma}
Suppose {\bf Assumption (M)} or {\bf Assumption (R)} holds. 
Given a family of displacements $\{y_h\in\H_0^1(S_h,\mathbb{R}^3)\}_{h\in(0,h_0)}$ with 
\beq
E_h(y_h):=\frac1{h}\int_{S_h}\W(\nabla y_h+I)\le Ch^\b,
\label{energy}
\eeq
consider the rescaling displacement $\eta_h$ defined on the universal manifold $U$ 
\[
\eta_h
=y_h\circ\pi_h
=W_h+w_h\n\in W^{1,2}(U,\mathbb{R}^3).
\]
Then the following holds.

{\bf Compactness:} There exist $W\in L^2(S,T)$, $w\in (H^1(S))'$ and a subsequence of $\{\eta_h\}$, still denoted by $\{\eta_h\}$, such that 
\[
h^{-\frac\b2}W_h\rightharpoonup W\circ\pi\quad \mbox{weakly in}\quad L^2(U,T),
\]
\[
h^{-\frac\b2}w_h\rightharpoonup w\circ\pi\quad \mbox{weakly in}\quad (H^1(U))',
\]
\[h^{-\frac\b2}(\sym DW_h+w_h\Pi)\rightharpoonup 
(\sym DW+w\Pi)\circ\pi \quad \mbox{weakly in}\quad [H^1(U,T^2_\sym)]',\]
and
\[h^{-\frac\b2}(\sym DW_h+w_h\Pi)\overset{*}{\rightharpoonup}
(\sym DW+w\Pi)\circ\pi \quad \mbox{weakly$^*$ in}\quad \M(U,T^2_\sym).\]

{\bf Lower Bound:} For the same subsequence in {\bf Compactness}, there holds 
\[
\liminf_{h\to 0^+}\frac{E_h(y_h)}{h^\b}\ge I_\b(W,w).
\]

{\bf Upper Bound:} Let $\b\in[0,2)\cup(\frac83,\infty)$. For each pair $(W,w)\in L^2(S,T)\times (H^1(S))'$, there exists $y_h\in\H^1_0(S_h)$ such that the convergences as shown in  {\bf Compactness} holds and 
\[
\lim_{h\to 0^+}\frac{E_h(y_h)}{h^\b}= I_\b(W,w).
\]

\end{thm}

\section{Proof of Rigidity Theorems}
Let $TS$ and $T^2S$ be collection of all the smooth tangent vector fields and 2nd-order tensor fields on $S$, respectively. For $A\in T^2S$, we define its divergence $\div A$ as follow:
\[
(\div A)(X)
=\div (i_XA)-\<A,DX\>\quad \mbox{for all}\quad X\in TS,
\]
where $i_XA\in TS$, $DX\in T^2S$ are defined as 
\[
\<i_XA,Y\>=A(X,Y)\quad \mbox{for all}\quad Y\in TS,
\]
\[
DX(Y,Z)=\<D_ZX,Y\>\quad \mbox{for all}\quad Y,Z\in TS,
\]
and $\div (i_XA)\in C^\infty(S)$ is the divergence of $i_XA\in TS$ on S.

Note that for all $f\in C^\infty(S)$, $X\in TS$, we have 
\beq
(\div A)(fX)
=&&\div (i_{fX}A)-\<A,D(fX)\>
\nonumber\\
=&&\div (fi_{X}A)-\<A,fDX+X\ot Df\>
\nonumber\\
=&&f\div (i_{X}A)+\<Df,i_{X}A\>
-f\<A,DX\>
-A(X,Df)
\nonumber\\
=&&(f\div A)(X).
\nonumber
\eeq
Hence, $\div A$ is a first-order tensor field on $S$, which we may regard as a tangent vector field on $S$ by the canonical isometry between $T_xS$ and its duel $T^*_xS$, $x\in S$, i.e., we have 
\beq
\<\div A,X\>
=\div (i_XA)-\<A,DX\>\quad \mbox{for all}\quad X\in TS.
\label{defdivA}
\eeq
One can check that the following formula holds 
\beq
\div (X\ot Y)=D_YX+(\div Y)X\quad \mbox{for all}\quad X,Y\in TS.
\label{f2}
\eeq

The main results of the present section are the following theorems. 
\begin{thm}\label{A}
(\ref{2d}) holds for $S=\a(\om)$ provided that 

\uppercase\expandafter{\romannumeral1}. $\a:\om\to M$ is an  asymptotic coordinate system.

\uppercase\expandafter{\romannumeral2}. $S$ has negative Gaussian curvature.

\uppercase\expandafter{\romannumeral3}. There exists an $A\in T^2S$ such that $A^T=A>0$, $\<A,\Pi\>=0$ and $\div A=0$.

\end{thm}
\begin{thm}\label{A3}
If the conditions in Theorem \ref{A} hold, 
then (\ref{3d}) holds for $S_h$.

\end{thm}
Theorem \ref{thm2d} and Theorem \ref{thm3d} immediately follow Theorem \ref{A} and Theorem \ref{A3}. 
\\
{\bf Proof of Theorem \ref{thm2d} and Theorem \ref{thm3d}: }

If {\bf Assumption (M)} holds, we directly take $A=\id$ by noting that $\<\id,\Pi\>=\tr\Pi=0$.
 
 Suppose that {\bf Assumption (R)} holds. By Lemma \ref{LemmaA2}, 
there exist $b\in C^\infty(S)$, $y_{\min}$, $y_{\max}\in C^\infty([0,1])$ and an asymptotic coordinate system $\a:\om\to S$ such that $S=\a(\om)$, 
\beq
\om=\bigcup_{x\in[0,1]}\{x\}\times[y_{\min}(x),y_{\max}(x)],\label{yminymax1}
\eeq
 and 
\beq
\a_x=\b_{s}-\frac{\Pi(\b_{s},\b_{s})}{2\Pi(\b_{s},\b_{t})}\b_{t},\quad \a_y=b\b_{t}.\label{ax1ax21}
\eeq

It suffices to prove the existence of $A$. 
Consider $A=u\a_x\ot u_x+v\a_y\ot \a_y$. 
Note that 
\[
\nabla_{\a_{y}}\a_{y}
=b\nabla_{\b_t}(b\b_t)
=b_t\a_{y}
=\Ga_{yy}^y\a_{y}
\]
by (\ref{ax1ax21}). 
Then (\ref{f2}) implies that 
\beq
\div A
=&&D_{\a_x}(u\a_x)
+u(\div\a_x)\a_x
+D_{\a_y}(v\a_y)
+v(\div\a_y)\a_y
\nonumber\\
=&&
[\a_{x}u+u(\Ga_{xx}^x+\div\a_{x})]\a_{x}
+[\a_{y}v
+v(\Ga_{yy}^y+\div\a_{y})+u\Ga_{xx}^y]\a_{y}
.\nonumber
\eeq
By (\ref{ax1ax21}) and Lemma \ref{LemmaA1}, there exist $u_1\in C^\infty(S)$ satisfying 
\[
\a_xu_1=-(\Ga_{xx}^x+\div \a_x).
\]
By (\ref{yminymax1}), there exist $v_1,w\in C^\infty(S)$ satisfying 
\[
\a_{y}v_1
=-(\Ga_{yy}^y+\div\a_{y})
\]
and 
\[
\a_{y}w
+w(\Ga_{yy}^y+\div\a_{y})+e^{u_1}\Ga_{xx}^y=0.
\]
Then $u=e^{u_1}$ and $v=ke^{v_1}+w$ satisfy the conditions of $A$ for sufficiently large $k\in\mathbb{N}^+$.

\qed

\subsection{Proof of Theorem \ref{A}}
In the present subsection, we always assume that the conditions in Theorem \ref{A} holds. 
There exists $a,b\in\mathbb{R}$ such that $\om\subset[a,b]^2$ since $\om\subset\mathbb{R}^2$ is compact. For any $y\in H^1_0(S,\mathbb{R}^3)$, we extend $y\circ \a$ by zero outside $\om$, and arbitrarily extend $\n\circ\a$ (and any other parameters needed) smoothly from $\om$ to $[a,b]^2$. 
Then Theorem \ref{A} and Theorem \ref{A3} holds if and only if (\ref{2d}) and (\ref{3d}) hold for $[a,b]^2$ under extended parameters. In the present subsection, without loss of generality, by rescaling, we may always suppose that 
\[\om=[0,1]^2.\]
Denote 
\[\partial_i=\a_{x_i},\quad i=1,2.\]

For any $y=W+w\n\in H^1_0(S,\mathbb{R}^3)$, 
let $u(x)=x+y(x)$, $x\in S$. 
Then $u_*=\id+y_*$ and 
\beq
u_*^Tu_*-\id
=&&(DW+w\nabla\n)^T(DW+w\nabla\n)+(Dw-\nabla_W\n)\ot (Dw-\nabla_W\n)
\nonumber\\
&&\quad+2(\sym DW+w\nabla\n).
\label{uu-id}
\eeq
Let $A\in T^2S$ be given satisfying the conditions in Theorem \ref{A}. Then (\ref{defdivA}) implies 
\beq
\<u_*^Tu_*-\id,A\>
=&&|(DW+w\nabla\n)A^\frac12|^2+|A^\frac12(Dw-\nabla_W\n)|^2
+2\div i_WA.
\nonumber
\eeq
Let 
\[\E=(u_*^Tu_*)^\frac12-\id.
\]
Then 
\[u_*^Tu_*-\id
=(\E+\id)^2-\id=\E^2+2\E.
\]
Thus, for any $y=W+w\n\in H^1_0(S,\mathbb{R}^3)$, we have 
\beq
\|DW+w\nabla\n\|^2_{L^2(S)}
+\|Dw-\nabla_W\n\|^2_{L^2(S)}
\le C(\|\E\|^2_{L^2(S)}+\|\E\|_{L^2(S)}).
\label{E}
\eeq
Theorem \ref{A} immediately follows (\ref{E}), provided $\|\E\|_{L^2(S)}\ge1$. It suffices to consider the case that $\|\E\|_{L^2(S)}\le 1$. Then (\ref{uu-id}) and (\ref{E}) yield
\beq
&&\|\sym DW+w\nabla\n\|_{L^1(S)}\nonumber\\
\le 
&&\|DW+w\nabla\n\|^2_{L^2(S)}
+\|Dw-\nabla_W\n\|^2_{L^2(S)}
+C\|u_*^Tu_*-\id\|_{L^1(S)}
\nonumber\\
\le&& C(\|\E\|^2_{L^2(S)}+\|\E\|_{L^2(S)})
\nonumber\\
\le&& C\|\E\|_{L^2(S)}.\qquad
\label{L1}
\eeq

It is easy to derive the $L^2$-regularity of $W$ on plates (referring to the proof of Theorem 2.1 in \cite{Conti}) by (\ref{L1}) and the celebrated Korn-Sobolev inequality 
\[\|W\|_{L^2(\om)}\le C\|\sym DW\|_{L^1(\om)}\quad
\mbox{for all}\quad W\in H^1_0(\om,\mathbb{R}^2),
\]
where $\om\subset \mathbb{R}^2$ is a given bounded region in $\mathbb{R}^2$, which won't happen in the case that second fundamental form $\Pi=\nabla\n$ appears, while the $L^1$-regularity of $W$ still holds. 

Denote $W_i=\<W_i,\partial_i\>$, $i=1,2$. 
Fix $\varphi_i\in C^\infty(S)$, $i=1,2$, such that 
\[
\partial_i\varphi_i=-\Ga_{ii}^i.
\]
Then 
\beq
\partial_1(e^{\varphi_1}W_1)
=&&e^{\varphi_1}(W_1\partial_1\varphi_1
+\partial_1W_1)\nonumber\\
=&&e^{\varphi_1}(-\Ga_{11}^1W_1
+\<D_{\partial_1}W,\partial_1\>
+\<W,D_{\partial_1}\partial_1\>
)\nonumber\\
=&&e^{\varphi_1}(\<\sym DW+w\Pi,\partial_1\ot\partial_1\>
+\Ga_{11}^2W_2
).\label{W11}
\eeq
Similarly, we have 
\beq
\partial_2(e^{\varphi_2}W_2)
=&&e^{\varphi_2}(\<\sym DW+w\Pi,\partial_2\ot\partial_2\>
+\Ga_{22}^1W_1
).\label{W22}
\eeq
\begin{lem}\label{WL1}There exists a constant $C>0$ such that for any $y=W+w\n\in H^1_0(S,\mathbb{R}^3)$, there holds 
\beq
\|W\|_{L^1(S)}\le C\|\sym DW+w\nabla\n\|_{L^1(S)}.\nonumber
\eeq

\end{lem}
\begin{proof}
By (\ref{W11}), we have 
\beq
|W_1|(x_1,x_2)
\le&&
C|e^{\varphi_1}W_1|(x_1,x_2)
\le C\int_0^{x_1}|\partial_1(e^{\varphi_1}W_1)|(s_1,x_2)ds_1\nonumber\\
\le&&C\int_0^{x_1}\big(|\sym DW+w\Pi|
+|W_2|\big)(s_1,x_2)ds_1.
\label{W1x2}
\eeq
Analogously, 
\beq
|W_2|(x_1,x_2)
\le&&C\int_0^{x_2}
\big(|\sym DW+w\Pi|+|W_1|\big)(x_1,s_2)ds_2.
\label{W2x1}
\eeq
For all $x_1\in[0,1]$, denote 
\[f(x_1)=\int_0^1|W_1|(x_1,x_2)dx_2.
\]
(\ref{W1x2}) and (\ref{W2x1}) yield  
\beq
f(x_1)
\le &&C\int_0^1\int_0^{x_1}\big(|\sym DW+w\Pi|
+|W_2|\big)(s_1,x_2)ds_1dx_2
\nonumber\\
\le &&C\|\sym DW+w\Pi\|_{L^1(S)}
+C\int_0^1\int_0^{x_1}\int_0^{x_2}
\big(|\sym DW+w\Pi|+|W_1|\big)(s_1,s_2)ds_2ds_1dx_2
\nonumber\\
\le &&C\|\sym DW+w\Pi\|_{L^1(S)}
+C\int_0^1\int_0^{x_1}\int_0^1
|W_1|(s_1,s_2)ds_2ds_1dx_2
\nonumber\\
\le &&C\|\sym DW+w\Pi\|_{L^1(S)}
+C\int_0^{x_1}
f(s_1)ds_1.
\nonumber
\eeq
Gronwall's inequality implies 
\[
f(x_1)\le 
C\|\sym DW+w\Pi\|_{L^1(S)}\quad\mbox{for all}\quad
x_1\in[0,1].
\]
Thus, 
\[
\|W_1\|_{L^1(S)}\le 
C\|\sym DW+w\Pi\|_{L^1(S)}.
\]
We finish the proof by obtaining the same estimates for $\|W_2\|_{L^1(S)}$.

\end{proof}

By (\ref{WL1}) and (\ref{L1}), we obtain there holds for any $y=W+w\n\in H^1_0(S,\mathbb{R}^3)$ that 
\beq
\|W\|_{L^1(S)}\le C\|\E\|_{L^2(S)}.\label{WL1'}
\eeq
To obtain the optimal exponent of thickness in rigidity estimates (\ref{3d}) for hyperbolic shells, $L^2$-regularity of $W$ as shown in (\ref{2d}) is necessary, referring to the linear elasticity theory. The following lemma takes a vital role in the subsequent improvement.

\begin{lem}\label{linear-non}
For any $B,C\in \mathbb{R}^{2\times2}$ with $C^T=C\ge0$, there holds 
\[
\<Be,e\>
\le |e|^2\Big|\big[(B+I_2)^T(B+I_2)+C\big]^\frac12-I_2\Big|
\]
for any $e\in\mathbb{R}^2$.
\end{lem}
\begin{proof}
Denote $D=B+I_2$. Suppose that $|e|=1$. Then 
\beq
\<De,e\>^2
\le&&|De|^2
\le \<(D^TD+C)e,e\>
=|(D^TD+C)^\frac12e|^2.
\nonumber
\eeq
We obtain 
\[
\<Be,e\>=\<De,e\>-1
\le |(D^TD+C)^\frac12e|-|e|
\le \big|(D^TD+C)^\frac12-I_2\big|.
\]
\end{proof}
Take $B=DW+w\nabla\n$ and $C=(Dw-\nabla_W\n)\ot (Dw-\nabla_W\n)$. Then 
\[(B+\id)^T(B+\id)+C=u_*^Tu_*.
\]
By Lemma \ref{linear-non}, for any $X\in TS$, we have
\beq
\<DW+w\Pi,X\ot X\>
\le |\E||X|^2.
\label{key}
\eeq

{\bf Proof of Theorem \ref{A}:}
By (\ref{W11}) and (\ref{key}), we have 
\beq
\partial_1(e^{\varphi_1}W_1)
=&&e^{\varphi_1}(\<DW+w\Pi,\partial_1\ot\partial_1\>
+\Ga_{11}^2W_2)
\le C(|\E|+|W_2|).
\nonumber
\eeq
Note that for any $(x_1,x_2)\in[0,1]^2$, there holds  
\[
e^{\varphi_1}W_1(x_1,x_2)
=\int_0^{x_1}[\partial_1(e^{\varphi_1}W_1)](s_1,x_2)ds_1,
\]
and 
\[
-e^{\varphi_1}W_1(x_1,x_2)
=\int_{x_1}^1[\partial_1(e^{\varphi_1}W_1)](s_1,x_2)ds_1.
\]
We obtain 
\beq
|W_1(x_1,x_2)|
\le C\int_0^1(|\E|+|W_2|)(s_1,x_2)ds_1,
\label{W1}
\eeq
and analogously, 
\beq
|W_2(x_1,x_2)|
\le C\int_0^1(|\E|+|W_1|)(x_1,s_2)ds_2.
\label{W2}
\eeq
Substituting (\ref{W2}) into (\ref{W1}), we obtain 
\beq
|W_1(x_1,x_2)|
\le C\int_0^1|\E|(s_1,x_2)ds_1
+C\int_0^1\int_0^1(|\E|+|W_1|)(s_1,s_2)ds_2ds_1.
\label{WL1'}
\eeq
(\ref{WL1'}) and Lemma \ref{WL1} yield 
\[
\|W_1\|_{L^2(S)}
\le C\|\E\|_{L^2(S)}
+C\|W_1\|_{L^1(S)}
\le C\|\E\|_{L^2(S)}.
\]
The same estimates also holds for $W_2$. Thus, 
\beq
\|W\|_{L^2(S)}
\le C\|\E\|_{L^2(S)}.
\label{WL2}
\eeq

Now we estimate norm $\|w\|_{(H^1(S))'}$. Fix $X_1,X_2\in TS$ such that 
\[\Pi(X_i,X_i)=(-1)^i,\quad i=1,2.\]
For any $\varphi\in C^1(S)$, let $\varphi_+=\max\{\varphi,0\}$, $\varphi_-=\max\{-\varphi,0\}$. Then $\varphi_+,\varphi_-\in H^1(S)$ and 
\[
D\varphi_+
=\chi_{\varphi>0}D\varphi,\quad 
D\varphi_-
=-\chi_{\varphi<0}D\varphi.
\]
For $i=1,2$, (\ref{key}) and (\ref{WL2}) imply
\beq
(-1)^i\int_S\varphi_+w
=&&\int_S\varphi_+w\Pi(X_i,X_i)
\nonumber\\
=&&\int_S\varphi_+
\<DW+w\Pi,X_i\ot X_i\>
-\<DW,\varphi_+X_i\ot X_i\>
\nonumber\\
\le &&\int_S\varphi_+
|X_i|^2|\E|
+\<W,\div(\varphi_+X_i\ot X_i)\>
\nonumber\\
\le &&C\|\E\|_{L^2(S)}\|\varphi\|_{H^1(S)}.
\nonumber
\eeq
Hence, 
\[\bigg|\int_S\varphi_+w\bigg|
\le C\|\E\|_{L^2(S)}\|\varphi\|_{H^1(S)}.
\]
Analogously, 
\[\bigg|\int_S\varphi_-w\bigg|
\le C\|\E\|_{L^2(S)}\|\varphi\|_{H^1(S)}.
\]
(\ref{2d}) follows. 

\qed

\subsection{Proof of Theorem \ref{A3}}
The component of nonlinear strain $(\nabla y)^T\nabla y+(\nabla y)^T+\nabla y$ acting on $\n\ot\n$ takes the form 
\[
w_t(w_t+2)+|W_t|^2,
\]
which is a hint to search for estimates of the following type. 

\begin{lem}\label{f(f+2)}
There exists constant $C>0$ such that for any $\tau\in[0,1]$, $f\in H^1_0([0,1]^2)$, there holds  
\beq
\|f\|^2_{L^2([0,1]^2)}
\le C\|f(f+2)\|_{L^1([0,1]^2)}
+C\|f(f+2)\|_{L^1([0,1]^2)}^\tau\|\nabla f\|^{2\tau}_{L^2([0,1]^2)}.
\label{wt}
\eeq

\end{lem}
\begin{proof}
Given $f\in C^1_0([0,1]^2)$, let 
\[A_1=\{x\in[0,1]^2:|f(x)|\le 1/2\},\quad
A_2=\{x\in[0,1]^2:|f(x)+2|\le 1/2\},
\]
and
\[
A_3=\{x\in[0,1]^2:|f(x)|>1/2,\quad|f(x)+2|>1/2\}.
\]
Clearly, $A_1$, $A_2$ and $A_3$ are mutually disjoint and $A_1\cup A_2\cup A_3=[0,1]^2$. It follows from the definition of $A_1$ and $A_3$ that $|f(x)+2|> 1/2$ on $A_1\cup A_3$, then 
\beq
\int_{A_1\cup A_3}f^2
\le \int_{A_1\cup A_3}|f(f+2)|+2|f|
\le C\|f(f+2)\|_{L^1(A_1\cup A_3)}
,\quad \label{A1A3}
\eeq
and 
\beq
|A_3|
\le 4\|f(f+2)\|_{L^1(A_3)}.\label{|A3|}
\eeq
Then (\ref{wt}) follows (\ref{A1A3}) and the following  
\beq
|A_2|\le C\|f(f+2)\|^\tau_{L^1([0,1]^2)}\|\nabla f\|^{2\tau}_{L^2([0,1]^2)}.
\label{A2}
\eeq
To show (\ref{A2}), let 
\[\pi_1(A_2)
=\{x_1\in[0,1]:\exists\ x_2\in [0,1],(x_1,x_2)\in A_2\},
\]
\[\pi_2(A_2)
=\{x_2\in[0,1]:\exists\ x_1\in [0,1],(x_1,x_2)\in A_2\}.
\]
For any $x_1\in \pi_1(A_2)$, let
\[
t_{-2}(x_1)
=\inf\{t\in[0,1]:|f(x_1,t)+2|\le1/2\}.
\]
Then $t_{-2}(x_1)>0$, since $f\in C^1_0([0,1]^2)$. Let 
\[
t_{0}(x_1)
=\sup\{t\in[0,t_{-2}(x_1)]:|f(x_1,t)|\le1/2\}<t_{-2}(x_1).
\]
Then\[
\{(x_1,x_2)\in[0,1]^2:x_1\in\pi_1(A_2),\ t_{0}(x_1)<x_2<t_{-2}(x_1)\}
\subset A_3.
\]
For each $x_1\in\pi_1(A_2)$, we have 
\beq
f(x_1,t_0(x_1))-f(x_1,t_{-2}(x_1))
=&&f(x_1,t_0(x_1))-[f(x_1,t_{-2}(x_1))+2]+2
\nonumber\\
\ge&&-\frac12-\frac12+2=1.
\nonumber
\eeq
Hence, (\ref{|A3|}) implies 
\beq
|\pi_1(A_2)|
\le&& \int_{\pi_1(A_2)}[f(x_1,t_0(x_1))-f(x_1,t_{-2}(x_1))]dx_1
\nonumber\\
=&& \int_{\pi_1(A_2)}\int_{t_0(x_1)}^{t_{-2}(x_1)}f_{x_2}(x_1,x_2)dx_2dx_1
\nonumber\\
\le&& \int_{A_3}|\nabla f(x_1,x_2)|dx_2dx_1
\le |A_3|^\frac12\|\nabla f\|_{L^2(A_3)}
\nonumber\\
\le&& C\|f(f+2)\|_{L^1([0,1]^2)}^\frac12\|\nabla f\|_{L^2([0,1]^2)}.
\nonumber
\eeq
Similarly, we have 
\[
|\pi_2(A_2)|
\le C\|f(f+2)\|_{L^1([0,1]^2)}^\frac12\|\nabla f\|_{L^2([0,1]^2)}.
\]
Therefore, by noting that $|A_2|\le 1$, we obtain for any $\tau\in[0,1]$, 
\[
|A_2|
\le |A_2|^\tau
\le \big(|\pi_1(A_2)||\pi_2(A_2)|\big)^\tau
\le C\|f(f+2)\|_{L^1([0,1]^2)}^\tau\|\nabla f\|^{2\tau}_{L^2([0,1]^2)}.
\]

\end{proof}

\begin{lem}\label{eigenvalue}
Let $A\in\mathbb{R}^{3\times3}$ and $B=A|_V:V\to\mathbb{R}^3$ be the restriction of $A$ on a 2-dimension linear subspace $V$ of $\mathbb{R}^3$, then 
\[
|(B^TB)^\frac12-\id_V|
\le \dist(A,SO(3)),
\]
where $\id_V:V\to V$ is the identity operator of $V$.

\end{lem}
\begin{proof}Suppose that $V=\mathbb{R}^2$ and $A=(B,b)$, $b\in\mathbb{R}^3$. Then 
\[
A^TA
=\left(\begin{matrix}
B^TB&&&B^Tb\\
b^TB&&&b^Tb
\end{matrix}
\right).
\]
Let $0\le\lambda_1\le \lambda_2\le\lambda_3$ be the eigenvalues of $A^TA$ and $0\le\mu_1\le \mu_2$ be the eigenvalues of $B^TB$. 
By Cauchy's interlace theorem, we have 
\[
\lambda_1\le\mu_1\le \lambda_2\le\mu_2\le\lambda_3.
\]
Let $a_i=(\sqrt{\lambda_i}-1)^2$, $i=1,2,3$. If $\max\{a_1,a_2\}=a_1$, then 
\[\max\{a_1,a_2\}
+\max\{a_2,a_3\}
\le a_1+a_2+a_3.
\]
If 
$\max\{a_1,a_2\}=a_2$, then $1\le \lambda_2\le \lambda_3$ and thus 
\[\max\{a_1,a_2\}
+\max\{a_2,a_3\}
=a_2+a_3.
\]
Hence, 
\beq
|(B^TB)^\frac12-I_2|^2
=&&(\sqrt{\mu_1}-1)^2+(\sqrt{\mu_2}-1)^2
\le\max\{a_1,a_2\}
+\max\{a_2,a_3\}
\nonumber\\
\le&& (\sqrt{\lambda_1}-1)^2+(\sqrt{\lambda_2}-1)^2
+(\sqrt{\lambda_3}-1)^2
\nonumber\\
=&& |(A^TA)^\frac12-I_3|^2
\le \dist^2(A,SO(3)).
\nonumber
\eeq

\end{proof}

Define rotation operator $Q:TS\to TS$ by 
\[QX=X\times\n,\quad \mbox{for all}\quad X\in TS.\]
One can check that $X\in TS$ is a closed vector field, i.e., $dX=0$, if and only if 
\[\div QX=0.\]

{\bf Proof of Theorem \ref{A3}: }Without loss of generality, we may extend $y\in \H^1_0(S_h,\mathbb{R}^3)$ to $y\in \H^1_0(\tilde S_h,\mathbb{R}^3)$ by zero for fixed $S\subset\subset S'\subset\subset\tilde S$. 

By Theorem 4.8 in \cite{Marta} (see also Theorem 6 in \cite{Muller2006}), for any $y\in\H^1_0(S_h,\mathbb{R}^3)$, there exists $R\in C^\infty(\tilde S,\mathbb{R}^{3\times3})$ such that 
\beq
R=I\quad\mbox{on}\quad \tilde S-S'.
\label{R-I}
\eeq
and
\beq
\|\nabla y+I-R\circ\pi\|^2_{L^2(\tilde S_h)}
\le C\varepsilon^2,\qquad
\label{0}
\eeq
\beq
h^3\|\nabla R\|^2_{L^2(\tilde S)}
+h^5\|\nabla^2 R\|^2_{L^2(\tilde S)}
+h^7\|\nabla^3 R\|^2_{L^2(\tilde S)}
\le C\varepsilon^2,\qquad
\label{1}
\eeq
where $\varepsilon=\|\dist(\nabla y+I,SO(3))\|_{L^2(S_h)}$ and $\pi:S_h\to S$, $x+t\n(x)\mapsto x$. It is known that 
\[
\pi_*:T_x\tilde S_h(\cong\mathbb{R}^3)\to T_x\tilde S,\quad \mbox{for all}\quad x\in \tilde S, 
\] 
is the orthogonal projection to $T\tilde S$. Let $\{e_1,e_2,e_3\}$ be the standard basis of $\mathbb{R}^3$. 
Then $\pi_*R^Te_i=R^Te_i-\<R^Te_i,\n\>\n\in C^\infty(\tilde S,T\tilde S)$ and there exists $\eta_i\in H^1_0(\tilde S)\cap C^\infty(\tilde S)$ such that 
\beq
\div D\eta_i=\div \pi_*(R^T-I)e_i\quad \mbox{on}\quad \tilde S.
\label{eta}
\eeq
Then $Q[\pi_*(R^T-I)e_i-D\eta_i]$ is a closed vector field on $\tilde S$. The simple connectivity of $\tilde S$ yields  there exists $\xi_i\in C^\infty(\tilde S)$ such that 
\beq
\pi_*(R^T-I)e_i=QD\xi_i+D\eta_i,\label{etai}
\eeq
and 
\[
\frac{\partial\xi_i}{\partial\nu}=0\quad\mbox{at}\quad\partial \tilde S.
\]

For $x\in\tilde S$, let  
\[
y_i(x)=\bbint_{-h/2}^{\ h/2}\<y,e_i\>(x,t)dt
:=\frac1{h}\int_{-h/2}^{h/2}\<y,e_i\>(x+t\n(x))dt.
\]
Then $\div QDy_i=0$ since $Dy_i$ is of course an exact vector field on $\tilde S$. 

For any $\a:(-\sigma,\sigma)\to \tilde S$, $\b(s,t)=\a(s)+t\n\circ\a(s)$, we have  
\beq
\a'(s)y_i
=&&\frac{\partial}{\partial s}\bbint_{-h/2}^{\ h/2}\<y,e_i\>\circ \b(s,t)dt
=\bbint_{-h/2}^{\ h/2}\<\nabla_{\b_s}y,e_i\>\circ \b(s,t)dt
\nonumber\\
=&&\bbint_{-h/2}^{\ h/2}\<\nabla y\big|_{\b(s,t)}(\id+t\nabla\n)\a'(s),e_i\>dt.
\nonumber
\eeq
Hence, 
\beq
Dy_i(x)
=&&\bbint_{-h/2}^{\ h/2}[(\id+t\nabla\n)\pi_*(\nabla y)^Te_i](x,t)dt.
\nonumber
\eeq
By (\ref{etai}), $\div QDy_i=0$ and (\ref{R-I}), $y_i|_{\tilde S\backslash S}=0$, we have 
\beq
\int_{\tilde S}|D\xi_i|^2
=&&-\int_{\tilde S}\xi_i\div D\xi_i
=\int_{\tilde S}\xi_i\div Q[\pi_*(R^T-I)e_i-Dy_i]
\nonumber\\
=&&\int_{\tilde S}\Big<QD\xi_i(x),\pi_*(R^T(x)-I)e_i
-\bbint_{-h/2}^{\ h/2}[(\id+t\nabla\n)\pi_*(\nabla y)^Te_i](x,t)dt\Big>dx
\nonumber\\
=&&\int_{\tilde S}\Big<QD\xi_i(x),
\frac1{h}\int_{-h/2}^{\ h/2}[(\id+t\nabla\n)\pi_*(R\circ\pi-I-\nabla y)^Te_i](x,t)dt
\Big>dx
\nonumber\\
\le&&
Ch^{-\frac12}\int_{\tilde S}|D\xi_i(x)|
\left(\int_{-h/2}^{\ h/2}|R\circ\pi-I-\nabla y|^2(x,t)dt\right)^\frac12
dx
\nonumber\\
\le &&Ch^{-\frac12}\|D\xi_i\|_{L^2(\tilde S)}
\|R\circ\pi-I-\nabla y\|_{L^2(\tilde S_h)}\nonumber\\
\le &&Ch^{-\frac12}\varepsilon\|D\xi_i\|_{L^2(\tilde S)}
,
\nonumber
\eeq
in view that 
\[
\pi_*(R^T(x)-I)e_i=\bbint_{-h/2}^{\ h/2}[(\id+t\nabla\n(x))\pi_*(R^T(x)-I)e_i]dt.
\]
We obtain 
\beq
\|D\xi_i\|_{L^2(\tilde S)}
\le Ch^{-\frac12}\varepsilon.
\label{Dxi}
\eeq

Define $\eta\in \H^1_0(\tilde S_h,\mathbb{R}^3)$ by 
\beq
\eta(x+t\n(x))
=\dsum_{i=1}^3\eta_i(x)e_i
+t[R(x)-I]\n(x).\label{defeta}
\eeq
For any $\a:(-\sigma,\sigma)\to \tilde S$, $\b(s)=\a(s)+t\n\circ\a(s)$, we have 
\[
\eta\circ \b(s)
=\sum_{i=1}^3
\eta_i\circ \a(s)e_i
+t[R\circ \a(s)-I]\n\circ \a(s)
.\]
By (\ref{etai}), 
\beq
\nabla_{\b'(s)}\eta
=&&\sum_{i=1}^3
[\a'(s)\eta_i]e_i
+t(\nabla_{\a'(s)}R)\n\circ \a(s)
+[R\circ \a(s)-I]t\nabla_{\a'(s)}\n
\nonumber\\
=&&\sum_{i=1}^3
\<\a'(s),(R^T-I)e_i-QD\xi_i\>e_i
+t(\nabla_{\a'(s)}R)\n\circ \a(s)
+[R\circ \a(s)-I][\b'(s)-\a'(s)]
\nonumber\\
=&&[R\circ\a(s)-I]\b'(s)
-\sum_{i=1}^3
\<\a'(s),QD\xi_i\>e_i
+t(\nabla_{\a'(s)}R)\n\circ \a(s).
\nonumber
\eeq
Hence, by also $(\nabla_{\n}\eta)\circ\b(s)
=[R\circ\a(s)-I]\n\circ \a(s)$, we obtain on $\tilde S_h$ that 
\[|\nabla \eta+I-R\circ\pi|
\le
C\left(\sum_{i=1}^3
|D\xi_i|
+t|\nabla R|\right)\circ\pi.
\]
(\ref{Dxi}) and (\ref{1}) imply that  
\beq
\|\nabla \eta+I-R\circ\pi\|^2_{L^2(\tilde S_h)}
\le
Ch\int_{\tilde S}\sum_{i=1}^3|D\xi_i|^2+h^2|\nabla R|^2
\le C\varepsilon^2.
\label{eta+I-R}
\eeq
By (\ref{0}) and the definition of $\varepsilon$, $\eta\in \H^1_0(\tilde S_h,\mathbb{R}^3)\cap C^\infty(\tilde S_h,\mathbb{R}^3)$ satisfy
\beq
\|\nabla \eta-\nabla y\|_{L^2(\tilde S_h)}
+\|\dist(\nabla \eta+I,SO(3))\|_{L^2(\tilde S_h)}
\le C\varepsilon.
\qquad
\label{eta-y}
\eeq
Let $w=\<\eta,\n\>$ and $W=\eta-w\n$.

{\bf Step 1:} Estimates for $\|Dw\|_{L^2(\tilde S_h)}$ and $\|DW\|_{L^2(\tilde S_h)}$. By the definition of $\eta$, it suffices to consider the restriction of $\eta$ on $\tilde S$, in view that  the tangential derivative of the remaining part still containing `` $t$ ".

Let $\xi (x)=\dsum_{i=1}^3\eta_i(x)e_i$ and $u(x)=x+\xi (x)
$, for all $x\in\tilde S$. Then 
\[
u(x)=x+t\n(x)+\eta(x+t\n(x))
-tR(x)\n(x),\quad\mbox{for all}\quad x\in\tilde S,\quad|t|\le h/2.
\]
For any $\a:(-\sigma,\sigma)\to \tilde S$, $\b(s)=\a(s)+t\n\circ\a(s)$, we have 
\beq
u_*\a'(s)
=&&\frac{d}{ds}u\circ\a(s)
=\frac{d}{ds}\big(\b(s)+\eta\circ \b(s)-tR\circ\a(s)\n\circ\a(s)\big)
\nonumber\\
=&&(I+\nabla\eta)\b'(s)
-t(\nabla_{\a'(s)}R)\n\circ\a(s)
-tR\circ\a(s)\nabla_{\a'(s)}\n
\nonumber\\
=&&(I+\nabla\eta)|_{\b(s)}\a'(s)
-t(\nabla_{\a'(s)}R)\n\circ\a(s)
-t(I+\nabla\eta-R\circ\pi)|_{\b(s)}\nabla_{\a'(s)}\n.
\nonumber
\eeq
By Lemma \ref{eigenvalue}, we obtain 
\[
|(u_*^Tu_*)^\frac12-\id|
\le \dist(\nabla \eta+I,SO(3))
+Ct|\nabla R|+Ct|\nabla \eta+I-R\circ\pi|.
\]
Hence, by (\ref{eta-y}), (\ref{1}) and (\ref{eta+I-R}), 
\beq
&&h\|(u_*^Tu_*)^\frac12-\id\|^2_{L^2(\tilde S)}
\nonumber\\
\le&& C\big(\|\dist(\nabla \eta+I,SO(3))\|^2_{L^2(\tilde S_h)}
+h^3\|\nabla R\|^2_{L^2(\tilde S)}
+h^2\|\nabla \eta+I-R\circ\pi\|^2_{L^2(\tilde S_h)}
\big)
\nonumber\\
\le&& C\varepsilon^2.
\nonumber
\eeq
Theorem \ref{A} yields that 
\beq
\|V\|_{L^2(\tilde S)}+\|v\|_{(H^1(\tilde S))'}
\le C\|(u_*^Tu_*)^\frac12-\id\|_{L^2(\tilde S)}
\le Ch^{-\frac12}\varepsilon,
\label{wH-1}
\eeq
where $v=\<\xi ,\n\>$, $V=\xi -v\n$. 
(\ref{1}), (\ref{eta}) and Poincar\'e's inequality imply that 
\beq
\|\xi \|_{H^2(\tilde S)}
\le&&
C\sum_{i=1}^3\|\eta_i\|_{H^2(\tilde S)}
\le
C(\|R-I\|_{L^2(S)}
+\|\nabla R\|_{L^2(S)})
\le
Ch^{-\frac32}\varepsilon.
\label{wH2}
\eeq
(\ref{wH-1}), (\ref{wH2}) and Sobolev interpolation inequality yield
\beq
\|Dv\|_{L^2(\tilde S)}
\le&&\|v\|^\frac13_{H^{-1}(\tilde S)}
\|v\|^\frac23_{H^2(\tilde S)}
\le C(h^{-\frac12}\varepsilon)^\frac13
(h^{-\frac32}\varepsilon)^\frac23
=Ch^{-\frac76}\varepsilon,
\nonumber
\eeq
and
\beq
\|DV\|_{L^2(\tilde S)}
\le&&\|V\|^\frac12_{L^2(\tilde S)}
\|V\|^\frac12_{H^2(\tilde S)}
\le C(h^{-\frac12}\varepsilon)^\frac12
(h^{-\frac32}\varepsilon)^\frac12
=Ch^{-1}\varepsilon.
\nonumber
\eeq
Hence, by (\ref{1}) and Poincar\'e's inquality, we obtain 
\[
\|Dw\|^2_{L^2(\tilde S_h)}
\le Ch\|Dv\|^2_{L^2(\tilde S)}
+Ch^3\|\nabla R\|^2_{L^2(\tilde S)}
+Ch^3\|R-I\|^2_{L^2(\tilde S)}
\le Ch^{-\frac43}\varepsilon^2,
\]
and 
\[
\|DW\|^2_{L^2(\tilde S_h)}
\le Ch\|DV\|^2_{L^2(\tilde S)}
+Ch^3\|\nabla R\|^2_{L^2(\tilde S)}
+Ch^3\|R-I\|^2_{L^2(\tilde S)}
\le Ch^{-1}\varepsilon^2.
\]

{\bf Step 2:} Estimates for $\|w_t\|_{L^2(\tilde S_h)}$ and $\|W_t\|_{L^2(\tilde S_h)}$.

Denote 
\[\E=[(\nabla\eta+I)^T(\nabla\eta+I)]^\frac12-I\quad\mbox{and}\quad
\E'=[(\nabla\eta+I)(\nabla\eta+I)^T]^\frac12-I.
\]
Let $\{E_1,E_2\}$ is an orthogonal basis of $T_xS$, $x\in S$. The matrix of $\nabla\eta$ takes the form 
\[
\nabla\eta
=\left(
\begin{matrix}
DW+w\nabla\n&&W_t\\
(Dw-\nabla_W\n)^T&&w_t
\end{matrix}
\right)\]
under basis $\{E_1,E_2,\n(x)\}$. We obtain 
\[
(\nabla\eta)^T(\nabla\eta)
=\left(
\begin{matrix}
(DW+w\nabla\n)^T(DW+w\nabla\n)
+(Dw-\nabla_W\n)(Dw-\nabla_W\n)^T
&&*\\
*&&w_t^2+|W_t|^2
\end{matrix}
\right),
\]
and 
\[
(\nabla\eta)
(\nabla\eta)^T
=\left(
\begin{matrix}
(DW+w\nabla\n)(DW+w\nabla\n)^T
+W_tW_t^T
&&*\\
*&&w_t^2+|Dw-\nabla_W\n|^2
\end{matrix}
\right).
\]
Let $\id:T_x\tilde S\to T_x\tilde S$ be the identity operator on $T_x\tilde S$, $x\in\tilde S$. Note that 
\[\E^2+2\E=(\nabla\eta)^T(\nabla\eta)+(\nabla\eta)^T+\nabla\eta,
\]
and
\[\E'^2+2\E'=(\nabla\eta)(\nabla\eta)^T+(\nabla\eta)^T+\nabla\eta.
\]
We obtain 
\[
\<\E^2+2\E-(\E'^2+2\E'),\id\>
=\<(\nabla\eta)^T(\nabla\eta)-(\nabla\eta)(\nabla\eta)^T,\id\>
=|Dw-\nabla_W\n|^2-|W_t|^2,
\]
and
\[
|w_t(w_t+2)|
\le |W_t|^2+|\E^2+2\E|.
\]
Hence, by $|\E|=|\E'|\le \dist(\nabla\eta+I,SO(3))$, (\ref{E}), Lemma \ref{eigenvalue} and H\"older's inequality, 
\beq
\|W_t\|^2_{L^2(\tilde S_h)}
+\|w_t(w_t+2)\|_{L^1(\tilde S_h)}
\le&& C\|Dw-\nabla_W\n\|^2_{L^2(\tilde S_h)}
+C\varepsilon^2+Ch^\frac12\varepsilon
\nonumber\\
\le&& C\varepsilon^2+Ch^\frac12\varepsilon.
\label{wt(wt+2)}
\eeq
By (\ref{R-I}), (\ref{1}) and (\ref{defeta}), we known that for each fixed $t\in(-h/2,h/2)$, 
\[
f(x):=w_t(x+t\n(x))
=\<[R(x)-I]\n(x),\n(x)\>
\in H^1_0(\tilde S),
\]
and 
\[
\|Df\|_{L^2(\tilde S)}
\le 
C\|R-I\|_{L^2(\tilde S)}
+C\|\nabla R\|_{L^2(\tilde S)}
\le 
C\|\nabla R\|_{L^2(\tilde S)}
\le Ch^{-\frac32}\varepsilon.
\]
Let $\tau\in[0,1]$. Lemma \ref{f(f+2)} and (\ref{wt(wt+2)}) yield 
\beq
\|f\|^2_{L^2(\tilde S)}
\le&& C\|f(f+2)\|_{L^1(\tilde S)}
+C\|f(f+2)\|^\tau_{L^1(\tilde S)}\|D f\|^{2\tau}_{L^2(\tilde S)}
\nonumber\\
\le&& Ch^{-1}(\varepsilon^2+h^\frac12\varepsilon)
+C[h^{-1}(\varepsilon+h^\frac12)\varepsilon]^\tau
(h^{-\frac32}\varepsilon)^{2\tau}
\nonumber\\
\le&& Ch^{-1}(\varepsilon^2+h^\frac12\varepsilon)
+Ch^{-\frac72\tau}(h^{-\frac12\tau}\varepsilon^\tau+1)
\varepsilon^{3\tau}
.\nonumber
\eeq
Integrating over $t\in(-h/2,h/2)$ and substituting into (\ref{wt(wt+2)}), we otain 
\[
\|W_t\|^2_{L^2(\tilde S_h)}
+\|w_t\|^2_{L^2(\tilde S_h)}
\le C\big(\varepsilon^2+h^\frac12\varepsilon
+h^{1-\frac72\tau}(h^{-\frac12\tau}\varepsilon^\tau+1)
\varepsilon^{3\tau}
\big).
\]
If $h^{-\frac12}\varepsilon\ge1$, let $\tau=\frac12$, we obtain 
\[
\|W_t\|^2_{L^2(\tilde S_h)}
+\|w_t\|^2_{L^2(\tilde S_h)}
\le C\big(h^\frac12\varepsilon
+h^{-1}\varepsilon^2
+h^{-\frac34}\varepsilon^\frac32
\big)
\le Ch^{-1}\varepsilon^2.
\]
If $h^{-\frac12}\varepsilon\le1$, let $\tau=\frac23$. Then 
\[
\|W_t\|^2_{L^2(\tilde S_h)}
+\|w_t\|^2_{L^2(\tilde S_h)}
\le C\big(h^\frac12\varepsilon
+h^{-\frac43}\varepsilon^2
\big).
\]
It suffices to consider the case $h^{-\frac43}\varepsilon^2
\le h^\frac12\varepsilon
$, i.e., $\varepsilon^2
\le h^{\frac{11}3}$. 

Claim:
\[
\|\nabla \eta\|_{L^\infty(\tilde S_h)}
\le C h^{\frac{5}{33}}.
\]
Then we finish the proof for sufficiently small $h>0$ by (\ref{eta-y}) and 
\beq
\|\nabla\eta\|_{L^2(\tilde S_h)}
\le
&&C\big(\|w_t\|_{L^2(\tilde S_h)}
+\|W_t\|_{L^2(\tilde S_h)}
+\|Dw\|_{L^2(\tilde S_h)}
+\|DW\|_{L^2(\tilde S_h)}\big)
\nonumber\\
\le 
&&C\big(\|w_t\|_{L^2(\tilde S_h)}
+\|W_t+Dw-\nabla_W\n\|_{L^2(\tilde S_h)}
+\|Dw\|_{L^2(\tilde S_h)}
+\|DW\|_{L^2(\tilde S_h)}\big)
\nonumber\\
\le 
&&C\|\sym \nabla \eta\|_{L^2(\tilde S_h)}
+Ch^{-\frac23}\varepsilon
\nonumber\\
\le 
&&C\|\dist(\nabla \eta+I,SO(3))\|_{L^2(\tilde S_h)}
+C\|\nabla \eta\|_{L^\infty(\tilde S_h)}
\|\nabla \eta\|_{L^2(\tilde S_h)}
+Ch^{-\frac23}\varepsilon
\nonumber\\
\le 
&&C h^{\frac{5}{33}}\|\nabla \eta\|_{L^2(\tilde S_h)}
+Ch^{-\frac23}\varepsilon.
\nonumber
\eeq
To prove the claim, by Sobolev embedding and interpolation inequality, we obtain 
\[
\|f\|_{L^\infty(\tilde S)}
\le C_\sigma\|f\|_{W^{1,2+\sigma}(\tilde S)}
\le C_\sigma\|f\|^{\frac{2-\sigma}{2+\sigma}}_{W^{1,2}(\tilde S)}
\|f\|^{\frac{2\sigma}{2+\sigma}}_{W^{1,4}(\tilde S)}
\le
C_\sigma
\|f\|^{\frac{2-\sigma}{2+\sigma}}_{W^{1,2}(\tilde S)}\|f\|^{\frac{2\sigma}{2+\sigma}}_{W^{2,2}(\tilde S)}
\]
for all $f\in W^{2,2}(\tilde S)$, where constant $\sigma>0$ is sufficiently small to be determined later.

By (\ref{eta}), (\ref{1}), (\ref{R-I}) and Poincar\'e's inequality, we obtain 
\beq
\|D\eta_i\|_{L^\infty(\tilde S)}
\le&&C\|\eta_i\|^{\frac{2-\sigma}{2+\sigma}}_{W^{2,2}(\tilde S)}\|\eta_i\|^{\frac{2\sigma}{2+\sigma}}_{W^{3,2}(\tilde S)}
\le C\|R-I\|^{\frac{2-\sigma}{2+\sigma}}_{W^{1,2}(\tilde S)}\|R-I\|^{\frac{2\sigma}{2+\sigma}}_{W^{2,2}(\tilde S)}
\nonumber\\
\le&&C\|\nabla R\|^{\frac{2-\sigma}{2+\sigma}}_{L^2(\tilde S)}
\|\nabla^2 R\|^{\frac{2\sigma}{2+\sigma}}_{L^2(\tilde S)}
\le Ch^{-\frac32\frac{2-\sigma}{2+\sigma}
-\frac52\frac{2\sigma}{2+\sigma}}\varepsilon
\nonumber\\
\le&&Ch^{-\frac{6+7\sigma}{2(2+\sigma)}}
\varepsilon,
\nonumber
\eeq
\[
h\|\nabla R\|_{L^\infty(\tilde S)}
\le Ch\|\nabla R\|^{\frac{2-\sigma}{2+\sigma}}_{W^{1,2}(\tilde S)}\|\nabla R\|^{\frac{2\sigma}{2+\sigma}}_{W^{2,2}(\tilde S)}
\le Ch^{1-\frac52\frac{2-\sigma}{2+\sigma}
-\frac72\frac{2\sigma}{2+\sigma}}
\varepsilon
=Ch^{-\frac{6+7\sigma}{2(2+\sigma)}}
\varepsilon,
\]
and
\beq
\|\eta_t\|_{L^\infty(\tilde S)}
=&&\|R-I\|_{L^\infty(\tilde S)}
\le C\|R-I\|^{\frac{2-\sigma}{2+\sigma}}_{W^{1,2}(\tilde S)}\|R-I\|^{\frac{2\sigma}{2+\sigma}}_{W^{2,2}(\tilde S)}
\le Ch^{-\frac{6+7\sigma}{2(2+\sigma)}}
\varepsilon.
\nonumber
\eeq
Hence, $\varepsilon^2
\le h^{\frac{11}3}$ yields 
\beq
\|\nabla \eta\|_{L^\infty(\tilde S_h)}
\le &&C\Big(
\|D\eta\|_{L^\infty(\tilde S_h)}
+\|\eta_t\|_{L^\infty(\tilde S_h)}\Big)
\nonumber\\
\le &&C\Big(
\sum_{i=1}^3\|D\eta_i\|_{L^\infty(\tilde S)}
+h\|\nabla R\|_{L^\infty(\tilde S)}
+\|R-I\|_{L^\infty(\tilde S)}\Big)
\nonumber\\
\le&&C h^{-\frac{6+7\sigma}{2(2+\sigma)}}
\varepsilon
\le C h^{\frac{2-5\sigma}{3(2+\sigma)}}.
\nonumber
\eeq
We finish the proof of claim by taking $\sigma=\frac15$.

\qed

\section{$\Ga$-limits}
Fix $h_0>0$ sufficiently small such that mapping
\[\pi:S_{h_0}\to S,\quad x+t\n(x)\mapsto x\]
is well-defined on $S_{h_0}$. Let $U=S\times[-1/2,1/2]$ be the product manifold of Riemannian manifolds $S$ and $[-1/2,1/2]$ equipped with metric
\[
\<(X_1,t_1\partial_t),(X_2,t_2\partial_t)\>=\<X_1,X_2\>+t_1t_2,\quad\mbox{for all}\quad 
(X_i,t_i)\in T_xS\times \mathbb{R}.
\]
Then the Riemannian connection $\D$ on $U$ takes the form 
\beq
\D_{X_1+t_1\partial_t}(X_2+t_2\partial_t)
=D_{X_1}X_2+t_1(\partial_tt_2)\partial_t.\label{connectionU}
\eeq
Define 
\[
\pi:S\times[-1/2,1/2]
\to S,\quad (x,t)\mapsto x,
\]
and
\[
\pi_h:S\times[-1/2,1/2]
\to S_h,\quad (x,t)\mapsto x+th\n(x).
\]
Given a family of displacements $\{y_h\in\H_0^1(S_h,\mathbb{R}^3)\}_{h\in(0,h_0)}$ with 
\beq
E_h(y_h):=\frac1{h}\int_{S_h}\W(\nabla y_h+I)\le Ch^\b,
\label{energy}
\eeq
consider the rescaling displacement $\eta_h$ defined on universal manifold $U$ 
\beq
\eta_h
=h^{-\frac\b2}y_h\circ\pi_h
=W_h+w_h\n,\label{etah}
\eeq
and denote 
\[
F_h
=(\nabla y_h+I)\circ\pi_h.
\]

Let $dz$ be the unit volume element of $S_h$ (under the dot product of $\mathbb{R}^3$) at $z=x+t\n(x)$ and $dx\wedge dt$ be the one of $U$ at $(x,t)$. 
Let $\{E_1,E_2\}$ be the orthonormal basis of $T_xS$, $x\in S$. 
Then $dz=dx\wedge dt=E_1\wedge E_2\wedge \n(x)$ and moreover 
\beq
&&\pi_h^*dz(E_1,E_2,\n(x))
=dz(\pi_{h*}E_1,\pi_{h*}E_2,\pi_{h*}\n(x))
\nonumber\\
=&&\det\Big[(\id+th\nabla\n(x))E_1\ \ \  
(\id+th\nabla\n(x))E_2\ \  \ 
h\n(x)\Big]
\nonumber\\
=&&h\det(\id+th\nabla\n(x)).
\nonumber
\eeq
We obtain 
\[
\pi_h^*dz=h\det(\id+th\nabla\n(x))dx\wedge dt,
\] and  change of variables formula for any $f\in L^1(S_h)$,
\beq
\frac1h\int_{S_h}f(z)dz
=&&\frac1h\int_{U}f\circ\pi_h(x,t)\pi_h^*dz
\nonumber\\
=&&\int_{-\frac12}^{\frac12}\int_S
f\circ\pi_h(x,t)\det(\id+th\nabla\n(x))dxdt
\nonumber\\
=&&\int_{-\frac12}^{\frac12}\int_S
f\circ\pi_h(x,t)dxdt
+O(h\|f\|_{L^1(S_h)}).
\label{Descartes}
\eeq
\beq
I_h(F_h):=\int_{-\frac12}^{\frac12}\int_S
\W(F_h(x,t))
dxdt\le Ch^\b.
\label{Ih}
\eeq
By (\ref{Descartes}), Theorem \ref{thm3d} and (\ref{Ih}), we obtain 
\beq
\int_{-\frac12}^\frac12\int_S
|\nabla y_h|^2\circ\pi_h(x,t)dxdt
\le&& Ch^{-\frac43}\int_{-\frac12}^\frac12\int_S
\dist^2(F_h(x,t),SO(3))dxdt
\le
Ch^{\b-\frac43}.\qquad
\label{yh}
\eeq

The proof of Theorem \ref{Gamma} mainly follows the strategies of \cite[Theorem 2.1]{Conti}.

{\bf Proof of Theorem \ref{Gamma}:}

{\bf Compactness:} 
Let $\overline{\eta}_h=\overline{W}_h+\overline{w}_h\n\in H^1_0(S,\mathbb{R}^3)$, where 
\[\overline{W}_h(x)
=\int_{-\frac12}^{\frac12}W_h(x,s)ds
,\quad
\overline{w}_h(x)
=\int_{-\frac12}^{\frac12}w_h(x,s)ds
.
\]
Note that 
\[
\partial_t\eta_h(x,t)
=\partial_t[y_h(x+th\n(x))]
=h\nabla_{\n}y_h\circ\pi_h(x,t).
\]
Then 
\beq
\|\eta_h-\overline{\eta}_h\circ\pi\|^2_{L^2(U)}
\le&&\int_U\int_{-\frac12}^{\frac12}
|\eta_h(x,t)-\eta_h(x,s)|^2dsdxdt
\le\int_{-\frac12}^{\frac12}\int_S\int_{-\frac12}^{\frac12}
\int_{-\frac12}^{\frac12}
|\partial_t\eta_h(x,\tau)|^2d\tau dsdxdt
\nonumber\\
\le&&h^2\int_S\int_{-\frac12}^{\frac12}
|\nabla_{\n}y_h\circ\pi_h(x,\tau)|^2d\tau dx
\le Ch^{2-\frac43+\b}
= Ch^{\frac23+\b}
.
\nonumber
\eeq
Fix $t\in[-1/2,1/2]$. Define 
\[
u_h(x)=x+\eta_h(x,t)=x+y_h\circ\pi_h(x,t).
\]
Then for any $\tau\in T_xS$, we have 
\beq
u_{h*}\tau
=\tau+(\nabla y_h\circ\pi_h)(\id+th\nabla\n(x))\tau
=&&
F_h(x,t)\tau+th[\nabla y_h\circ\pi_h(x,t)]\nabla\n(x)\tau.\qquad
\label{uh*}
\eeq
By (\ref{uh*}), Lemma \ref{eigenvalue} and (\ref{yh}), we obtain 
\beq
&&\int_{-\frac12}^\frac12\int_S|(u_{h*}^Tu_{h*})^\frac12-\id|^2(x,t)dxdt
\nonumber\\
\le&&C
\int_{-\frac12}^\frac12\int_S\dist^2(F_h(x,t)+th[\nabla y_h\circ\pi_h(x,t)]\nabla\n(x),SO(3))
dxdt
\nonumber\\
\le&&C
\int_{-\frac12}^\frac12\int_S\dist^2(F_h(x,t),SO(3))
+h^2|\nabla y_h|^2\circ\pi_h(x,t)dxdt
\nonumber\\
\le&&
Ch^{\b}
+Ch^{\frac23+\b}
\le 
Ch^{\b}.
\label{uTu-I}
\eeq
Hence, Theorem \ref{thm2d} yields 
\beq
\|\overline{W}_h\|^2_{L^2(S)}
\le&&\int_{-\frac12}^\frac12\int_S|W_h|^2(x,t)dxdt
\le C\int_{-\frac12}^\frac12\int_S|(u_{h*}^Tu_{h*})^\frac12-\id|^2(x,t)dxdt
\le 
Ch^{\b},
\nonumber
\eeq
\beq
\bigg|\int_S\varphi \overline{w}_hdx\bigg|
\le&&\int_{-\frac12}^\frac12\bigg|\int_S\varphi(x)w_h(x,t)dx\bigg|dt
\nonumber\\
\le&& C\|\varphi\|_{H^1(S)}
\int_{-\frac12}^\frac12\bigg|\int_S|(u_{h*}^Tu_{h*})^\frac12-\id|^2(x,t)dx\bigg|^\frac12dt
\nonumber\\
\le&&C\|\varphi\|_{H^1(S)}
h^\frac\b2\qquad\qquad\qquad
\mbox{for any}\qquad \varphi\in H^1(S),
\nonumber
\eeq
and by (\ref{L1}), (\ref{yh}), (\ref{connectionU}), 
\beq
&&\|\sym D\overline{W}_h+\overline{w}_h\Pi\|_{L^1(S)}
\nonumber\\
\le&& C\int_{-\frac12}^\frac12\int_S|\sym DW_h+w_h\Pi|(x,t)
dxdt
\nonumber\\
\le&& C\int_{-\frac12}^\frac12\bigg|\int_S|(u_{h*}^Tu_{h*})^\frac12-\id|^2(x,t)dx\bigg|^\frac12dt
\le Ch^{\frac\b2}.
\nonumber
\eeq
Thus, $(h^{-\frac\b2}\overline{W}_h,h^{-\frac\b2}\overline{w}_h,h^{-\frac\b2}(\sym D\overline{W}_h+\overline{w}_h\nabla\n))$ is bounded in $L^2(S,T)\times (H^1(S))'\times L^1(S,T^2_{\sym})$.

{\bf Lower Bound:} Let $G_h=h^{-\frac\b2}[(F_h^TF_h)^\frac12-I]\in L^2(U,\mathbb{R}^{3\times3}_{\sym})$ and 
\[
\overline{G}_h
=\int_{-\frac12}^\frac12G_h(\cdot,t)dt\in L^2(S,\mathbb{R}^{3\times3}_{\sym}).
\]
$\overline{G}_h$ is uniformly bounded in $L^2(S,\mathbb{R}^{3\times3}_{\sym})$, since $|G_h|=h^{-\frac\b2}\dist(F_h,SO(3))$. Taking a subsequence, we obtain 
\beq
\overline{G}_h\rightharpoonup G\quad\mbox{weakly in}\quad L^2(S,\mathbb{R}^{3\times3}_{\sym}).
\label{toG}
\eeq
Let $G_{\tan}\in T^2S$ be the restriction of $G$ on $TS$, i.e., 
\[
G_{\tan}(\tau,\mu)=\<G\tau,\mu\>\quad\mbox{for all}\quad \tau,\mu\in T_xS,\quad x\in S.
\]
We conclude 
\[\liminf_{h\to 0^+}\frac{I_h(F_h)}{h^\b}
\ge \frac12\int_S\Q_2(G_{\tan}(x))dx
\]
by the standard methods in nonlinear elastic shell theory, referring to \cite[$Part$ II of the proof of Theorem 2.1]{Conti} for clamped plates, or \cite[Theorem 6.1]{Muller2002} and \cite[Theorem 5.10]{Marta} for plates and shells without clamped lateral boundary conditions, respectively. 

It suffices to relate $G_{\tan}$ with $W$ and $w$. Note that 
\[
F_h^TF_h
=(h^{\frac\b2}G_h+I)^2
=h^\b G^2_h+2h^{\frac\b2}G_h+I.
\]
We obtain 
\beq
\left\|\frac12h^{-\frac\b2}(F_h^TF_h-I)-G_h\right\|_{L^1(U)}
=\frac12h^{\frac\b2} \|G_h\|^2_{L^2(U)}
\le Ch^{\frac\b2}.\label{FTF-Gh}
\eeq
Fix $t\in[-1/2,1/2]$, for any $\tau,\mu\in T_xS$, (\ref{uh*})yields 
\[|\<F_h\tau,F_h\mu\>-\<u_{h*}\tau,
u_{h*}\mu\>|
\le C|\tau||\mu|\big(h|\nabla y_h|+h^2|\nabla y_h|\big)\circ\pi_h.
\]
By (\ref{yh}), 
\beq
h^{-\frac\b2}\|(F_h^TF_h-I)_{\tan}-(u_{h*}^Tu_{h*}-\id)\|_{L^1(U)}
\le C\big(h^{\frac13}+h^{\frac23+\frac\b2}\big).\label{FTF-uTu}
\eeq
(\ref{uu-id}) yields 
\beq
\frac12h^{-\frac\b2}(u_{h*}^Tu_{h*}-\id)
=\frac12h^{-\frac\b2}A_h+h^{-\frac\b2}(\sym DW_h+w_h\nabla\n),\label{u*Tu*-I}
\eeq
where 
\[A_h=(DW_h+w_h\nabla\n)^T(DW_h+w_h\nabla\n)
+(Dw_h-\nabla_{W_h}\n)\ot (Dw_h-\nabla_{W_h}\n).\]
If $\b>\frac83$, it follows from (\ref{yh}) that 
\beq
h^{-\frac\b2}\|A_h\|_{L^1(U)}
\le&&Ch^{-\frac\b2}
\int_{-\frac12}^\frac12\int_S
|\nabla y_h|^2\circ\pi_h(x,t)dxdt
\le Ch^{\frac\b2-\frac43}.
\label{Ah>8/3}
\eeq
Hence, we conclude from (\ref{toG}) and (\ref{FTF-Gh})--(\ref{Ah>8/3}) that 
\[
G_{\tan}=\sym DW+w\Pi\in L^2(S,T^2_{\sym}).
\]
If $0\le \b\le\frac83$, then (\ref{E}) and (\ref{uTu-I}) imply 
\beq
h^{-\frac\b2}\|A_h\|_{L^1(U)}
\le&&Ch^{-\frac\b2}\int_{-\frac12}^\frac12
\left(\int_S
\big(|DW_h+w_h\nabla\n|^2+|Dw_h-\nabla_{W_h}\n|^2\big)(x,t)dx\right)dt
\nonumber\\
\le&&Ch^{-\frac\b2}
\int_{-\frac12}^\frac12\bigg|\int_S|(u_{h*}^Tu_{h*})^\frac12-\id|^2(x,t)dx\bigg|^\frac12dt
\le C.
\nonumber
\eeq
Therefore, $h^{-\frac\b2}\dint_{-\frac12}^\frac12A_h(\cdot,t)dt$ is uniformly bounded in $L^1(S,T^2_{\sym})$, then converges weakly$^*$ to a Radon measure $\mu\in \M(S,T^2_+)$ by taking a subsequence. We obtain 
\[
G_{\tan}dx=\sym DW+w\Pi+\mu.
\]

{\bf Upper Bound:} Let $X_1=\frac{Q\partial_2}{\<\partial_1,Q\partial_2\>}$, $X_2=-\frac{Q\partial_1}{\<\partial_1,Q\partial_2\>}$. Then  $\{X_1,X_2\}$ is the conjugate basis of $\{\partial_1,\partial_2\}$, i.e., $\<X_i,\partial_j\>=\delta_{ij}$, $i,j=1,2$. 

Analogous to the proof of \cite[Theorem 2.1]{Conti}, just need to consider the case that $W\in C^\infty_0(S,T)$, $w\in C^\infty_0(S)$, $a\in C^\infty_0(S,T)$, $b\in C^\infty_0(S)$ and $\mu=\sum_{i,j=1}^2m_{ij}X_i\ot X_j$, where $M=(m_{ij})\in \mathbb{R}^{2\times2}_+$. 

{\bf Case 1:} $\b>\frac83$. In this case, $M=0$. Let $y_h=h^{\frac\b2}\tilde y_h\in\H^1_0(S_h,\mathbb{R}^3)$, where 
\[\tilde y_h(x+t\n(x))
=(W+w\n)(x)+t(a+b\n-Dw+\nabla\n W)(x)
\quad\mbox{for all}\quad x+t\n(x)\in S_h.
\]
Let $\{E_1,E_2\}$ be an orthonormal basis of $T_xS$, $x\in S$. Then the matrix of $\nabla\tilde y_h|_{x+t\n(x)}$ under basis $\{E_1,E_2,\n(x)\}$ takes the form 
\[
\nabla \tilde y_h|_{x+t\n(x)}
=
\left(
\begin{matrix}
DW+w\nabla\n+tD(a-Dw+\nabla\n W)+tb\nabla\n&&&&a-Dw+\nabla\n W\\
[Dw-\nabla\n W-t\nabla\n (a-Dw+\nabla\n W)+tDb]^T&&&&b
\end{matrix}
\right).
\]
Hence, for $(x,t)\in U=S\times[-1/2,1/2]$, we have 
\beq
h^{-\frac\b2}(F_h^TF_h-I)(x,t)
=&&(h^{\frac\b2}\nabla^T \tilde y_h\nabla \tilde y_h
+\nabla^T\tilde y_h+\nabla\tilde y_h)\circ\pi_h(x,t)
\nonumber\\
=&&
\left(
\begin{matrix}
2\sym (DW+w\nabla\n)&&&&
a
\\
a^T&&&&2b
\end{matrix}
\right)(x)+O(h)+O(h^{\frac\b2}).
\nonumber
\eeq
We obtain 
\[
\frac12h^{-\frac\b2}(F_h^TF_h-I)
\to\left(
\begin{matrix}
\sym (DW+w\nabla\n)&&&&
\frac12a
\\
\frac12a^T&&&&b
\end{matrix}
\right)\circ\pi\quad\mbox{strongly in}\quad L^2(U,\mathbb{R}^{3\times3}_{\sym}).
\]

{\bf Case 2:} $0\le \b<2$.  Let $\rho\in C^\infty(\mathbb{R})$ satisfy $\rho\ge0$, $\supp \rho\subset [-1,1]$ and $\int_{\mathbb{R}}\rho(s)ds=1$. Let $\rho_\sigma(s)=\sigma^{-1}\rho(\sigma^{-1}s)$ and 
\[
\zeta(s)
=
\begin{cases}
s,\quad &s\in [k,k+\frac12),\ k\in\mathbb{Z},\\
1-s,\quad &s\in [k+\frac12,k+1),\ k\in\mathbb{Z},
\end{cases}
\]
be the triangular wave. For $\delta=h^{\frac\b4}|\ln h|^{-1}$, define  
\[
\zeta_\delta(s)
=\delta(\rho_{|\ln h|^{-1}}*\zeta)(\delta^{-1}s).
\]
Then $\zeta_\delta$ is $\delta$-periodic and $\zeta_\delta\overset{*}{\rightharpoonup}
0$ weakly$^*$ in $W^{1,\infty}(\mathbb{R})$, as $h\to 0^+$.

Let $R=(r_{ij})\in O(2)$ such that 
\[
M=R^T\diag(\lambda_1,\lambda_2)R
.
\]
Then 
\[
\mu
=\sum_{i,j=1}^2m_{ij}X_i\ot X_j
=\sum_{i=1}^2
\lambda_ia_i\ot a_i,
\]
where $a_i=r_{i1}X_1+r_{i2}X_2$. Define $b_i=\lambda_1^\frac12r_{1i}+\lambda_2^\frac12r_{2i}$, $i=1,2$, and 
\[
\psi_\delta \circ\a(x_1,x_2)
=\zeta_\delta (b_1x_1+b_2x_2)
(\lambda_1^\frac12a_1-\lambda_2^\frac12a_2).
\]
Then 
\[
D_{\partial_i}\psi_\delta 
=b_i\zeta_\delta '(b_1x_1+b_2x_2)
(\lambda_1^\frac12a_1-\lambda_2^\frac12a_2)
+O(\delta).
\]
By $\<X_i,\partial_j\>=\delta_{ij}$ and $b_1X_1+b_2X_2=\lambda_1^\frac12a_1
+\lambda_2^\frac12a_2$, we obtain 
\beq
D\psi_\delta 
=&&\zeta_\delta '(b_1x_1+b_2x_2)
(\lambda_1^\frac12a_1-\lambda_2^\frac12a_2)
\ot (b_1X_1+b_2X_2)
+O(\delta)
\nonumber\\
=&&\zeta_\delta '(b_1x_1+b_2x_2)
(\lambda_1^\frac12a_1-\lambda_2^\frac12a_2)
\ot (\lambda_1^\frac12a_1
+\lambda_2^\frac12a_2)
+O(\delta).
\nonumber
\eeq
In particular, 
\[
\sym D\psi_\delta
=\begin{cases}
\lambda_1a_1\ot a_1-\lambda_2a_2\ot a_2+O(|\ln h|^{-1})
,\qquad \mbox{if}\quad(x_1,x_2)\in \Om^+,\\
\lambda_2a_2\ot a_2-\lambda_1a_1\ot a_1+O(|\ln h|^{-1})
,\qquad \mbox{if}\quad(x_1,x_2)\in \Om^-,\\
\end{cases}
\]
where
\[
\Om^+
=\left\{(x_1,x_2)\in [0,1]^2:\ \delta^{-1}(b_1x_1+b_2x_2)\in \left[k,k+\frac12\right),\ k\in\mathbb{Z}\right\},
\]
and
\[
\Om^-=[0,1]^2\backslash \Om^+.
\]
Define 
\[
\varphi_\delta \circ\a(x_1,x_2)
=\begin{cases}
\zeta_\delta (2\lambda_2^\frac12(r_{21}x_1+r_{22}x_2)),&\quad \mbox{if}\quad(x_1,x_2)\in\Om^+,\\
\zeta_\delta (-2\lambda_1^\frac12(r_{11}x_1+r_{12}x_2)),&\quad \mbox{if}\quad(x_1,x_2)\in\Om^-.
\end{cases}
\]
Then $\varphi_\delta \in W^{1,\infty}(S)$ since if $(x_1,x_2)\in \overline{\Om^+}\cap\overline{\Om^-}$, then 
\[
k=2\delta^{-1}(b_1x_1+b_2x_2)
\in \mathbb{Z}.
\]
Thus, 
\[
\zeta_\delta (2\lambda_2^\frac12(r_{21}x_1+r_{22}x_2))
=
\zeta_\delta (2\lambda_2^\frac12(r_{21}x_1+r_{22}x_2)
-\delta k)
=
\zeta_\delta (-2\lambda_1^\frac12(r_{11}x_1+r_{12}x_2)).
\]
Furthermore, 
\[
D\varphi_\delta \circ\a(x_1,x_2)
=\begin{cases}
2\lambda_2^\frac12\zeta_\delta '(2\lambda_2^\frac12(r_{21}x_1+r_{22}x_2))a_2,&\quad \mbox{if}\quad(x_1,x_2)\in\Om^+,\\
-2\lambda_1^\frac12\zeta_\delta '(-2\lambda_1^\frac12(r_{11}x_1+r_{12}x_2))a_1,&\quad \mbox{if}\quad(x_1,x_2)\in\Om^-,
\end{cases}
\]
and
\[
(D\varphi_\delta \ot D\varphi_\delta )\circ\a(x_1,x_2)
=\begin{cases}
4\lambda_2a_2\ot a_2+O(|\ln h|^{-1}),&\quad \mbox{if}\quad(x_1,x_2)\in\Om^+,\\
4\lambda_1a_1\ot a_1+O(|\ln h|^{-1}),&\quad \mbox{if}\quad(x_1,x_2)\in\Om^-.
\end{cases}
\]
Therefore, 
\[
\sym D\psi_\delta +D\varphi_\delta \ot D\varphi_\delta =\mu+O(|\ln h|^{-1})\quad a.e. \quad x\in S.
\]
Applying Lemma \ref{approx} to $\varphi_\delta$ with $r=\sqrt2$,  $\sigma=h^{\frac\b4}|\ln h|^{-3}$, 
$v=\delta^{-1}(b_1,b_2)$, $p\ge2$, we obtain $\tilde \varphi_h$ such that 
\beq
\|\tilde \varphi_h-\varphi_{\delta(h)}\|_{L^p(S)}
\le&&
C|\ln h|^{-3}h^{\frac\b4},\label{phi0}
\\
\|\nabla\tilde \varphi_h-\nabla \varphi_{\delta(h)}\|_{L^p(S)}
\le&&
C|\ln h|^{-\frac2p},\label{phi1}
\\
\|\nabla^2\tilde \varphi_h\|_{L^p(S)}
\le &&
C|\ln h|^3
h^{-\frac\b4}.\label{phi2}
\eeq
For $x\in S$, $t\in(-h/2,h/2)$, define $y_h(x+t\n(x))=h^{\frac\b2}\tilde y_h(x+t\n(x))$, where 
\[\tilde y_h(x+t\n(x))
=(W_h+w_h\n)(x)
+t(\xi_h+\zeta_h\n)(x)
,\]
and
\beq
W_h=&&W+\psi_{\delta(h)},
\nonumber\\
w_h=&&w+h^{-\frac\b4}\tilde\varphi_h,
\nonumber\\
\xi_h
=&&-h^{-\frac\b4}D\tilde \varphi_h-Dw+\nabla\n W+a,
\nonumber\\
\zeta_h
=&&-\frac12|D\tilde \varphi_h|^2+b.\nonumber
\eeq
Then for $x\in S$, $t\in(-1/2,1/2)$, we have 
\beq
\nabla \tilde y_h\circ\pi_h(x,t)
=&&
\left(
\begin{matrix}
DW_h+w_h\nabla\n+th(D\xi_h+\zeta_h\nabla\n)&&&&\xi_h\\
[Dw_h-\nabla\n W_h+th(D\zeta_h-\nabla\n \xi_h)]^T&&&&\zeta_h
\end{matrix}
\right)
\nonumber\\
=&&
A_0+A_1+thA_2+h^{-\frac\b4}A_3
-\frac12|D\tilde \varphi_h|^2\n\ot\n,
\nonumber
\eeq
where 
\beq
A_0(x)
=&&
\left(
\begin{matrix}
DW+w\nabla\n+D\psi_{\delta}
&&&&
-Dw+\nabla\n W+a
\\
Dw-\nabla\n W
&&&&b
\end{matrix}
\right),
\nonumber\\
A_1(x,t)
=&&
\left(
\begin{matrix}
h^{-\frac\b4}\tilde\varphi_h\nabla\n
-th^{1-\frac\b4}D^2\tilde \varphi_h
-\frac12th|D\tilde \varphi_h|^2\nabla\n
&&&&
0
\\
-\nabla\n \psi_{\delta}
-th(D^2\tilde \varphi_h)D\tilde \varphi_h
+th^{1-\frac\b4}\nabla\n D\tilde \varphi_h
&&&&0
\end{matrix}
\right),
\nonumber\\
A_2(x)
=&&
\left(
\begin{matrix}
-D(Dw-\nabla\n W-a)+b\nabla\n
&&&&0\\
[Db
+\nabla\n (Dw-\nabla\n W-a)]^T
&&&&0
\end{matrix}
\right),
\nonumber\\
A_3(x)
=&&
\left(
\begin{matrix}
0&&&&
-D\tilde \varphi_h
\\
(D\tilde\varphi_h)^T
&&&&0
\end{matrix}
\right).
\nonumber
\eeq
By (\ref{phi0}), (\ref{phi1}) and (\ref{phi2}), we have $\nabla \tilde y_h\circ\pi_h$ is uniformly bounded in $L^\infty(U,\mathbb{R}^{3\times3}_{\sym})$ and
\[
\|A_1(\cdot,t)\|_{L^4(S)}
\le 
C(|\ln h|^{-1}
+h^{1-\frac\b2}|\ln h|^3
+h^{1-\frac\b4}),\quad\mbox{for all}\quad t\in(-1/2,1/2).
\]
Let $B_h=A_0+A_1+thA_2
+\frac12|D\tilde \varphi_h|^2\n\ot\n$. Then $B_h$ is uniformly bounded in $L^4(U,\mathbb{R}^{3\times3}_{\sym})$ for all sufficiently small $h$. We have 
\beq
&&\frac12h^{-\frac\b2}(F_h^TF_h-I)(x,t)
\nonumber\\
=&&(\frac12h^{\frac\b2}\nabla^T \tilde y_h\nabla \tilde y_h
+\sym \nabla\tilde y_h)\circ\pi_h(x,t)
\nonumber\\
=&&
\frac12h^{\frac\b2}(B_h+h^{-\frac\b4}A_3)^T
(B_h+h^{-\frac\b4}A_3)
\nonumber\\
&&
+\left(
\begin{matrix}
\sym DW+w\Pi+\sym D\psi_{\delta}
&&&&
\frac12a
\\
\frac12a^T
&&&&-\frac12|D\tilde \varphi_h|^2+b
\end{matrix}
\right)
+\sym(A_1+thA_2)\nonumber\\
=&&
\left(
\begin{matrix}
\sym DW+w\Pi
+\sym D\psi_{\delta}+\frac12D\tilde \varphi_h\ot D\tilde \varphi_h
&&&&
\frac12a
\\
\frac12a^T
&&&&b
\end{matrix}
\right)
\nonumber\\
&&
+h^{\frac\b2}(\frac12B_h^TB_h+\sym A_3^TB_h)
+\sym(A_1+thA_2).
\nonumber
\eeq
We conclude 
\[
\frac12h^{-\frac\b2}(F_h^TF_h-I)
\to \left(
\begin{matrix}
\sym DW+w\Pi
+\mu
&&&&
\frac12a
\\
\frac12a^T
&&&&b
\end{matrix}
\right)
\quad\mbox{strongly in}\quad L^2(U,\mathbb{R}^{3\times3}_{\sym}).
\]
In view that $\|\nabla y_h\|_{L^\infty}=h^\frac\b2\|\nabla \tilde y_h\|_{L^\infty}\le Ch^\frac\b2$, we obtain 
\[
\dist(F_h,SO(3))
=\dist(I+\nabla y_h\circ\pi_h,SO(3))
\le C(|\nabla y_h|+|\nabla y_h|^2)\le Ch^\frac\b2.
\]
By \cite[Lemma A.1 and Lemma A.6]{Conti}, there exists a measurable mapping $R_h:U\to SO(3)$ such that $R_h^TF_h-I$ is symmetric and  moreover,
\[
|R_h^TF_h-I|=|F_h-R_h|=\dist(F_h,SO(3))\le Ch^\frac\b2.
\]
Let $G_h=h^{-\frac\b2}[R_h^TF_h-I]$. Then $F_h=R_h(h^{\frac\b2}G_h+I)$ and 
\[
\left\|G_h-\frac12h^{-\frac\b2}(F_h^TF_h-I)
\right\|_{L^2(U,\mathbb{R}^{3\times3}_{\sym})}
=\frac12h^\frac\b2\|G_h\|^2_{L^2(U,\mathbb{R}^{3\times3}_{\sym})}
\le Ch^\frac\b2.\]
Therefore, since $\Q_3$ is a quadratic form, we have 
\beq
\frac1{h^\b}\int_U\W(F_h)dxdt
=&&
\frac1{h^\b}\int_U\W(h^\frac\b2G_h+I)dxdt
\nonumber\\
=&&
\frac1{h^\b}\int_U\frac12\Q_3(h^\frac\b2G_h)
+o(h^\b |G_h|^2)dxdt
\nonumber\\
=&&
\int_U\left(\frac12\Q_3(G_h)
+\frac{o(h^\b |G_h|^2)}{h^\b|G_h|^2}|G_h|^2\right)dxdt
\nonumber\\
\to&&
\frac12\int_{-\frac12}^\frac12\int_S\Q_3\left(
\begin{matrix}
\sym DW+w\Pi
+\mu
&&&&
\frac12a
\\
\frac12a^T
&&&&b
\end{matrix}
\right)(x)
dxdt
\nonumber\\
=&&
\frac12\int_S\Q_2(\sym DW+w\Pi+\mu)
dx.
\nonumber
\eeq

\qed

\appendix 
\section{Asymptotic Coordinate System on Ruled Surfaces} 
Let $\b(s,t)=\sigma(s)+t\delta(s)$, $s\in[-\varepsilon_0,1+\varepsilon_0]$, $t\in\mathbb{R}$ and $M=\b([-\varepsilon_0,1+\varepsilon_0]\times\mathbb{R})$ be a smooth ruled surface with negative Gaussian curvature. We shall construct an asymptotic coordinate system covering $S=\b([0,1]^2)$. 

Let $a=-\frac{\Pi(\b_{s},\b_{s})}{2\Pi(\b_{s},\b_{t})}$, $X=\b_{s}+a\b_{t}$. Then 
\[
\Pi(X,X)
=\Pi(\b_{s},\b_{s})+2a\Pi(\b_{s},\b_{t})
+a^2\Pi(\b_{t},\b_{t})=0.
\]
\begin{lem}\label{LemmaA1}For any $g\in C^\infty(S)$, there exists $f\in C^\infty(S)$ such that 
\beq
Xf=g.\label{Xf}
\eeq
\end{lem}
\begin{proof}
There Fix $x\in[0,1]$. By the compactness of $[0,1]$, there exists $\varepsilon>0$ such that 
\[\xi'(\tau)
=a\circ\b(\tau,\xi(\tau)),\quad \xi(x)=t
\]
possesses a unique smooth solution $\xi(\tau;x,t)$ for all $(\tau,t)\in \overline{I\times J}$, where $I=(x-\varepsilon,x+\varepsilon)$, $J=(-\varepsilon,1+\varepsilon)$. 
Define $\varphi(s,t)=\b(s,\xi(s;x,t))$. Then 
\[
\varphi_{s}(s,t)=\b_s(s,\xi(s;x,t))
+\frac{\partial\xi(s;x,t)}{\partial s}
\b_t(s,\xi(s;x,t))
=X\circ\varphi(s,t),
\]
and 
\[
\varphi(x,t)=\b(x,\xi(x;x,t))=\b(x,t),
\]
for all $(s,t)\in I\times J$. 
If $\varphi(s_1,t_1)=\varphi(s_2,t_2)$ for some $(s_1,t_1),(s_2,t_2)\in I\times J$, then  
\[
\b(s_1,\xi(s_1;x,t_1))
=\b(s_2,\xi(s_2;x,t_2)).
\]
Thus, $s_1=s_2$ and $\xi(s_1;x,t_1)=\xi(s_2;x,t_2)=\xi(s_1;x,t_2)$. By the invertibility and uniqueness of the solutions of ODE, we obtain $t_1=t_2$. Therefore, $\varphi:I\times J\to U:=\varphi(I\times J)$ is a diffeomorphism. 

By the compactness of $S$, there exists finitely many $(x_{i},\varepsilon_i,\varphi_i,I_i,J_i,U_i)$ such that 
\[
S\subset\cup_{i=1}^nU_i,
\]
and
\[\varphi_i(s,t)=\b(s,\xi(s;x_{i},t))
\quad\mbox{for all}\quad
(s,t)\in I_i\times J_i,
\]
\[0\le x_1< x_2< \cdots<x_n\le 1,
\quad x_{i+1}-\varepsilon_{i+1}\in I_i
.
\]
\begin{figure}
\centering
\includegraphics[width=0.5\textwidth]{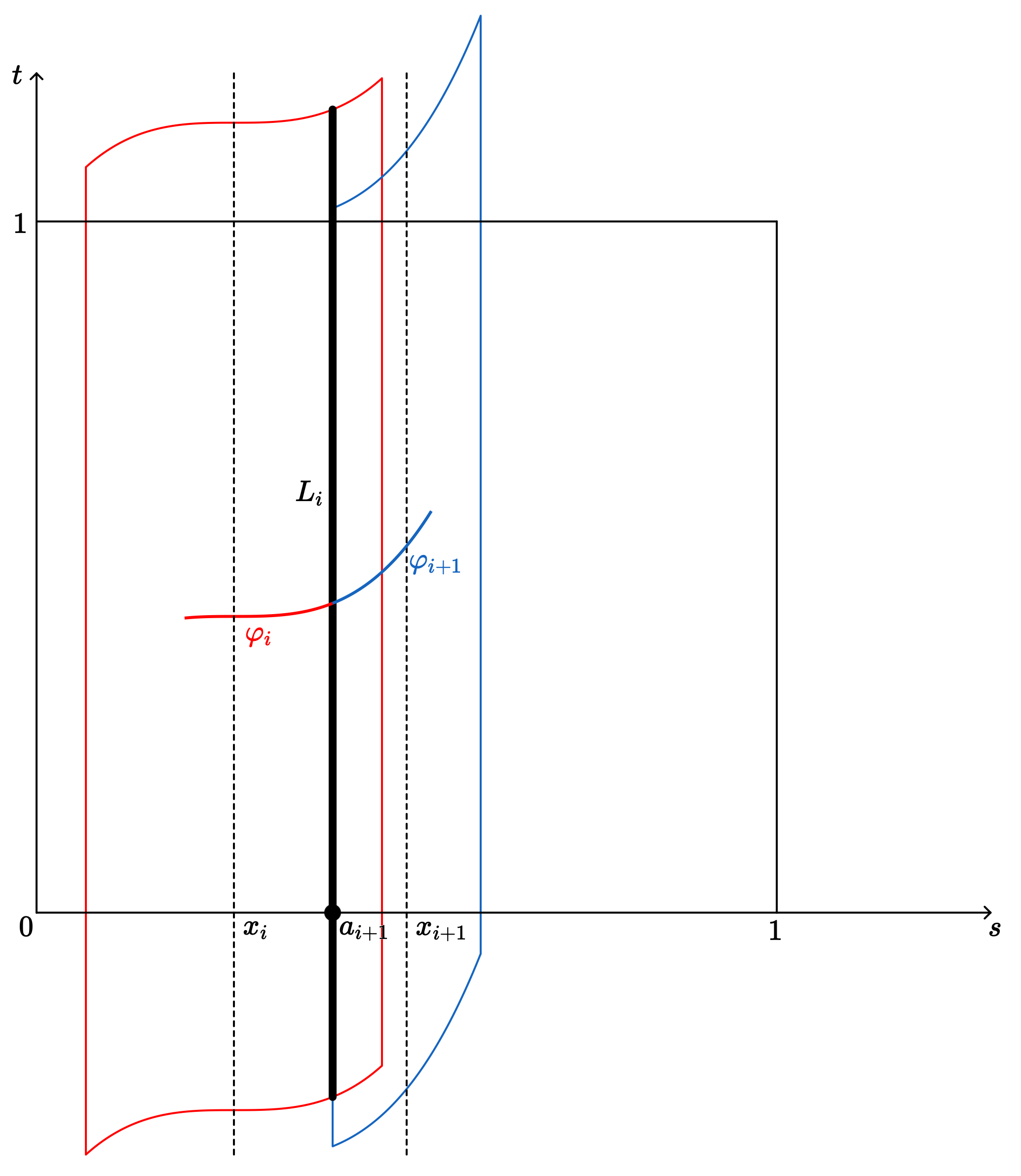}
\end{figure}
Let $V_i=\cup_{j=1}^i\overline {U_i}$ and $a_i=x_{i}-\varepsilon_i$, $i=1,\cdots,n$. 
We define $f$ on $S$ by induction. 

For $i=1$, let 
\[
f\circ\varphi_1(s,t)
=\int_0^{s}g\circ\varphi_1(\tau,t)d\tau,
\]
for any $(s,t)\in\overline{I_1\times J_1}$. 
Then (\ref{Xf}) holds on $V_1$. 

Suppose that $f$ has been defined on $V_i$ satisfying (\ref{Xf}). Denote 
$f_i=f|_{\overline{U_i}}$. 
For $i+1$, 
\[
L_i=\xi(a_{i+1};x_{i},\overline{J_i})
\]
is a closed interval by the continuity of $\xi(a_{i+1};x_{i},\cdot)$ on $\overline{J_i}=[-\varepsilon_i,1+\varepsilon_i]$. 
Note that 
\[
\b(a_{i+1},L_i)=\b(a_{i+1},\xi(a_{i+1};x_{i},\overline{J_i}))
=\varphi_i(a_{i+1},\overline{J_i})\subset \overline{U_i}
\]
since $a_{i+1}\in I_i$. 
Smoothly extend $f_i\circ\b(a_{i+1},\cdot)|_{L_i}$ to some $\bar f_i\in C^\infty(\mathbb{R})$. 

For all  $(s,t)\in \overline{I_{i+1}\times J_{i+1}}$, define 
\[
f_{i+1}\circ\varphi_{i+1}(s,t)
=\bar f_i(\xi(a_{i+1};x_{i+1},t))
+\int_{a_{i+1}}^{s}g\circ\varphi_{i+1}(\tau,t)d\tau.
\]
Then $Xf_{i+1}=g$ holds on $\overline{U_{i+1}}$. 
If $t\in \overline{J_{i+1}}$ satisfy $\xi(a_{i+1};x_{i+1},t)\in L_i$, then 
\beq
f_{i+1}\circ\b(a_{i+1},\xi(a_{i+1};x_{i+1},t))
=&&f_{i+1}\circ\varphi_{i+1}(a_{i+1},t)
=\bar f_i(\xi(a_{i+1};x_{i+1},t))
\nonumber\\
=&&f_i\circ\b(a_{i+1},\xi(a_{i+1};x_{i+1},t)).
\nonumber
\eeq
Hence, $f_i=f_{i+1}$ on $\overline{U_i}\cap \overline{U_{i+1}}$ by the uniqueness of solutions of ODE. Extending $f$ on $\overline{U_{i+1}}$ by $f_{i+1}$, we finish the induction. 
\end{proof}

\begin{lem}\label{LemmaA2}
There exist $b\in C^\infty(S)$, $b\ne0$ on $S$, $y_{\min}$, $y_{\max}\in C^\infty([0,1])$ and an asymptotic coordinate system $\a:\om\to S$ such that $S=\a(\om)$, 
\beq
\om=\bigcup_{x\in[0,1]}\{x\}\times[y_{\min}(x),y_{\max}(x)],\label{yminymax}
\eeq
 and 
\[
\a_x=X,\quad \a_y=b\b_{t}.
\]
\end{lem}
\begin{proof}
By Lemma \ref{LemmaA1}, there exists $f\in C^\infty(S)$ such that $\div(e^fX)=e^f(Xf+\div X)=0$. Then $e^fQX$ is a closed vector field on $S$. By the simply connectivity of $S$, there exists $y\in C^\infty(S)$ such that 
\[Dy=e^fQX=e^fQ(\b_{s}+a\b_{t}).
\]
We may assume $\<Q\b_{s},\b_{t}\>>0$ on $S$. 
For each $s\in[0,1]$, $h\circ\b(s,\cdot)$ increases on $[0,1]$ since 
\[
\frac{\partial}{\partial t}h\circ\b(s,t)
=\<Dh,\b_{t}\>\circ\b(s,t)
=b^{-1}\circ\b(s,t)
>0,
\]
where 
$b^{-1}=e^f\<Q\b_{s},\b_{t}\>$.

Let  $x(s,t)=s$. Define  
\[\psi:S\to \mathbb{R}^2, \quad 
\b(s,t)\mapsto (x(s,t),y\circ\b(s,t))\quad\mbox{for all}\quad(s,t)\in[0,1]^2.
\]
Then $\om:=\psi(S)$ satisfies (\ref{yminymax}) with 
\[y_{\min}(x)=y\circ\b(x,0),\quad
y_{\max}(x)=y\circ\b(x,1),\quad x\in[0,1].
\]
$\psi$ is locally diffeomorphic, in view that  
\[
\left(
\begin{matrix}
\psi_*\b_{s}\\
\psi_*\b_{t}
\end{matrix}
\right)
=\left(
\begin{matrix}
e_1+\<Dh,\b_{s}\>e_2\\
\<Dh,\b_{t}\>e_2
\end{matrix}
\right)
=\left(
\begin{matrix}
1&&-ab^{-1}\\
0&&b^{-1}
\end{matrix}
\right)
\left(
\begin{matrix}
e_1\\
e_2
\end{matrix}
\right),
\]
where $\{e_1,e_2\}$ is the standard basis of $\mathbb{R}^2$. 
$\psi$ is globally injective by the monotonicity of $y\circ\b(s,\cdot)$. 
Therefore, $\psi:S\to \om$ is diffeomorphic and 
$\a=\psi^{-1}:\om\to S$ is asymptotic by noting that 
\[
\left(
\begin{matrix}
\a_x\\
\a_y
\end{matrix}
\right)
=\left(
\begin{matrix}
\a_*e_1\\
\a_*e_2
\end{matrix}
\right)
=\left(
\begin{matrix}
(\psi_*)^{-1}e_1\\
(\psi_*)^{-1}e_2
\end{matrix}
\right)
=\left(
\begin{matrix}
1&&&a\\
0&&&b
\end{matrix}
\right)\left(
\begin{matrix}
\b_{s}\\
\b_{t}
\end{matrix}
\right)
=\left(
\begin{matrix}
X\\
b\b_{t}
\end{matrix}
\right).
\]

\end{proof}

\section{An Approximation Lemma}
Let $v\in\mathbb{R}^2$, $v\ne0$. Let 
\[
\Om^+
=\{x\in\mathbb{R}^2:\<x,v\>\in(k,k+1/2),\ k\in\mathbb{Z}\},
\]
and
\[
\Om^-
=\{x\in\mathbb{R}^2:\<x,v\>\in(k+1/2,k+1),\ k\in\mathbb{Z}\}.
\]
Then $\Ga=\partial\Om^+=\partial\Om^-$. Define 
\[
\H_v(\mathbb{R}^2)
=\{f^+\chi_{\Om^+}+f^-\chi_{\Om^-}:f^+\in W^{2,\infty}(\Om^+),\ f^-\in W^{2,\infty}(\Om^-),\ f^+=f^-\ \mbox{at}\ \Ga
\}.
\]

\begin{lem}\label{approx}
There exists a constant $C>0$ such that for any $r,\sigma>0$, $p\ge1$, $v\in\mathbb{R}^2$, $v\ne0$ with $4\sigma<|v|^{-1}<r$ and any $f\in \H_v(\mathbb{R}^2)\subset W^{1,\infty}(\mathbb{R}^2)$, there exists $\tilde f\in C^\infty(\mathbb{R}^2)$ such that 
\[
\|\tilde f-f\|_{L^p(B_r)}
\le
\sigma
\|\nabla f\|_{L^p(B_{2r})},
\]
\beq
\|\nabla\tilde f-\nabla f\|_{L^p(B_r)}
\le
C(r^2|v|\sigma)^\frac1p\|\nabla f\|_{L^\infty(B_{2r})}
+C\sigma
\big(\|\nabla^2 f^+\|_{L^p(B_{2r}^+)}
+\|\nabla^2 f^-\|_{L^p(B_{2r}^-)}
\big),\qquad\quad\label{nablatildef}
\eeq
\beq
\|\nabla^2\tilde f\|_{L^p(B_r)}
\le C(r^2|v|\sigma)^\frac1p\sigma^{-1}
\|\nabla f\|_{L^\infty(B_{2r})}
+C\big(\|\nabla^2 f^+\|_{L^p(B_{2r}^+)}
+\|\nabla^2 f^-\|_{L^p(B_{2r}^-)}
\big),\qquad\label{nabla2tildef}
\eeq
where $B_r=\{x\in\mathbb{R}^2:|x|<r\}$, $B_r^+=B_r\cap\Om^+$, $B_r^-=B_r\cap\Om^-$.

\end{lem}
\begin{proof}Suppose that $v=(a,0)$, $a>0$. Let $\rho\in C^\infty(\mathbb{R}^2)$ satisfy $\rho\ge0$, $\supp \rho\subset B_1$ and $\int_{\mathbb{R}^2}\rho(x)dx=1$. 
Given $f=f^+\chi_{\Om^+}+f^-\chi_{\Om^-}\in \H_v(\mathbb{R}^2)$, let 
\[
\tilde f=\rho_\sigma*f\in C^\infty(\mathbb{R}^2),
\]
where $\rho_\sigma(x)=\sigma^{-2}\rho(\sigma^{-1}x)$. Let $\frac1{p'}=1-\frac1p$. Then 
\beq
\|\tilde f-f\|^p_{L^p(B_r)}
\le&&
\int_{B_r}\left|
\int_{\mathbb{R}^2}
\rho_\sigma(y)^{\frac1{p'}}
\rho_\sigma(y)^\frac1p |f(x-y)-f(x)|dy\right|^pdx
\nonumber\\
\le&&
\int_{B_r}\int_{B_\sigma}\rho_\sigma(y) 
\left|\int_0^1\<\nabla f(x-ty),y\>dt\right|^pdydx
\nonumber\\
\le&&
\sigma^p
\int_0^1\int_{B_\sigma}\rho_\sigma(y) 
\int_{B_r}|\nabla f(x-ty)|^pdxdydt
\nonumber\\
\le&&
\sigma^p
\|\nabla f\|^p_{L^p(B_{2r})}.
\label{L2tildef}
\eeq
Define 
\[
\Ga_\sigma
=\bigcup_{k\in\mathbb{Z}}\left[\frac{k}{2a}-\sigma,\frac{k}{2a}+\sigma\right]\times\mathbb{R}.
\]
Then 
\[
|B_r\cap\Ga_\sigma|\le Cr\sigma\cdot r/a^{-1}=Cr^2a\sigma,\]
and thus, 
\[
\|\nabla\tilde f-\nabla f\|^p_{L^p(B_r\cap\Ga_\sigma)}
\le
C|B_r\cap\Ga_\sigma|\|\nabla f\|^p_{L^\infty(B_{2r})}
\le
Cr^2a\sigma \|\nabla f\|^p_{L^\infty(B_{2r})}.
\]
An argument analogous to (\ref{L2tildef}) yields 
\[
\|\nabla\tilde f-\nabla f\|^p_{L^p(B_r\backslash\Ga_\sigma)}
\le
\sigma^p
\big(\|\nabla^2 f^+\|^p_{L^p(B_{2r}^+)}
+\|\nabla^2 f^-\|^p_{L^p(B_{2r}^-)}
\big).
\]
(\ref{nablatildef}) follows. 
If $x\in B_r^+\backslash\Ga_\sigma$, then 
\[
\nabla^2 \tilde f(x)
=
\int_{B_\sigma}\rho_\sigma(y)\nabla^2f^+(x-y)dy. 
\]
We obtain 
\[
\|\nabla^2\tilde f\|^p_{L^p(B_r^+\backslash\Ga_\sigma)}
\le \|\nabla^2f^+\|^p_{L^p(B_{2r}^+)}.
\]
Similarly, 
\[
\|\nabla^2\tilde f\|^p_{L^p(B_r^-\backslash\Ga_\sigma)}
\le \|\nabla^2f^-\|^p_{L^p(B_{2r}^-)}.
\]
If $x\in B_r\cap \Ga_\sigma$, then 
\beq
|\tilde f_{x_ix_j}(x)|
=&&\left|\partial_{x_i}
\int_{\mathbb{R}^2}\rho_\sigma(y)f_{x_j}(x-y)dy\right|
=\sigma^{-2}\left|\partial_{x_i}
\int_{B_\sigma}\rho(\sigma^{-1}(x-y))f_{x_j}(y)dy\right|
\nonumber\\
=&&\sigma^{-3}\left|
\int_{B_\sigma}\rho_{x_i}(\sigma^{-1}(x-y))f_{x_j}(y)dy\right|
\le C\sigma^{-1}\|\nabla f\|_{L^\infty(B_{2r})}.
\nonumber
\eeq
Hence, 
\[
\|\nabla^2\tilde f\|^p_{L^p(B_r\cap \Ga_\sigma)}
\le C|B_r\cap\Ga_\sigma|\cdot\sigma^{-p}\|\nabla f\|^p_{L^\infty(B_{2r})}
\le Cr^2a\sigma^{1-p}\|\nabla f\|^p_{L^\infty(B_{2r})}.
\]
We obtain (\ref{nabla2tildef}).

\end{proof}

 \end{document}